\documentclass[11pts,final]{amsart}
\pdfoutput=1
\usepackage{amssymb}
\usepackage{epsfig}
\usepackage{multirow}
\usepackage{multicol}
\usepackage{hhline}
\usepackage{float}
\usepackage{color}
\usepackage{bbm}
\usepackage{yhmath}

\usepackage{amscd}
\usepackage{graphicx}

\DeclareGraphicsExtensions{.pdf,.png,.jpg}
\newtheorem{theorem}{Theorem}[section]
\newtheorem{prop}[theorem]{Proposition}
\newtheorem{cor}[theorem]{Corollary}
\newtheorem{lem}[theorem]{Lemma}
\newtheorem{defn}[theorem]{Definition}

\theoremstyle{remark}
\newtheorem{rem}[theorem]{Remark}

\def\o{\otimes}
\def\b{\beta}

\def\a{\alpha}
\def\mc{\mathcal}
\def\lan{\langle}
\def\ran{\rangle}

\def\d{\delta}

\def\g{\gamma}

\def\&{\,\,\mbox{and}\,\,}

\def\S{\Sigma}
\def\i{\infty}
\def\G{\Gamma}
\def\Z{\mathbb{Z}}

\def\RP1{\mathbb{R}P^1}
\def\Nn{\mathbb{N}_n}

\begin{document}
\title{Fukaya category of surfaces and mapping class group action}
\author{Haniya Azam$^{1}$, Christian Blanchet$^{2}$}
\address{$^{1}$Department of Mathematics, SBA School of Science and Engineering, LUMS, Opposite sector U, DHA, Lahore-54792, Pakistan}
\email {haniya.azam@lums.edu.pk}
\address{$^{2}$ Univ Paris Diderot, IMJ-PRG, UMR 7586 CNRS,  F-75013, Paris, France}
\email{Christian.Blanchet@imj-prg.fr}


\maketitle
\pagestyle{myheadings} \markboth{\centerline {\scriptsize
H. AZAM, C. BLANCHET }}{\centerline {\scriptsize
Fukaya category of surfaces and mapping class group action}}

\textwidth=13cm

\pagestyle{myheadings}

\begin{abstract}
We construct the  Fukaya category of a surface with genus greater than one
and compute its Grothendieck group. In this paper we consider a topological variant, in which we disregard the area form and use instead an admissibility condition borrowed from Heegaard-Floer theory which ensures invariance under isotopy. 
 We also study action of the mapping class group of the surface on its Fukaya category and establish faithfulness.
\end{abstract}
\maketitle
\section{Introduction}

The Fukaya category was introduced by Fukaya in his work on Morse homotopy, $A_{\infty}$-categories and Floer homologies in \cite{fukaya1993morse}. Kontsevich's homological mirror symmetry conjecture \cite{kontsevich1994homological}, kindled wide interest in the Fukaya category which relates it to the category of coherent sheaves on a mirror manifold.

Let $\S$ be a closed connected surface of genus higher than one.
The Fukaya category of $\S$ equipped with an area form has been constructed by Abouzaid in \cite{abouzaid2008fukaya}
 who  also defined a derived extension using twisted complexes and computed its Grothendieck group.
	 Our treatment here follows Abouzaid's work, but without using the area form. We give a topological variant of the Fukaya category and show that the count of objects in the moduli space is finite, in the absence of the symplectic form.

For the genus two curve studied by Seidel in \cite{seidel2011homological} from the perspective of homological mirror symmetry, the Fukaya category is defined using `balanced curves'. Both constructions of the Fukaya category of a surface $\S$ in \cite{abouzaid2008fukaya} and \cite{seidel2011homological} use an auxilliary choice of a one-form on the unit tangent bundle of $\S$, whose differential equals the pullback of the symplectic form. In our topological variant we use `admissible closed curves' which makes it a model intrinsic to the topology of the surface, independent of any auxilliary choices. This ensures topological invariance and functoriality without further discussion.

 De Silva, Robbin and Salamon in their memoir \cite{de2014combinatorial} consider what they call combinatorial Floer homology, using non-isotopic curves and the differential counting smooth lunes. They give a one-one correspondence between smooth lunes (upto homotopy) and index one holomorphic strips (upto translation) showing that the combinatorial and analytic descriptions of Floer homology agree. The Floer homology for admissible curves defined in this paper is more general than their setup with non-isotopic curves. In addition to this we also describe higher products and show that all this coincides with the analytic Floer-Fukaya definitions.

 Basic objects in this paper are
oriented immersed curves which are {\em unobstructed} (that is, such a curve lifts to a properly embedded line in the universal cover);
the sign of intersection defines an absolute $\Z/2\Z$-grading on all Floer complexes.
We describe the Grothendieck group of the topological version of the derived Fukaya category of $\S$ and prove the following theorem:
\begin{theorem}
For a closed oriented surface $\S$ of genus greater than one, the Grothendieck group $K_0(Fuk(\S))$ of the derived topological Fukaya category of $\S$ is isomorphic to $$H_1(S\S;\mathbb{Z}),$$ where $S\S$ is the unit tangent bundle of $\S$.
\end{theorem}
 We also study the action of the mapping class group $MCG(\S)$ on our topological Fukaya category $Fuk(\S)$ by functors and prove the following faithfulness result inspired by the work of Khovanov and Seidel in \cite{khovanov2002quivers}.

\begin{theorem}
Let $f\in MCG(\S)$ and suppose that the following isomorphism holds for any `admissible'
closed curves $\g_1 \& \g_2$,
$$HF(f(\g_1),\g_2)\cong HF(\g_1,\g_2).$$
Then $f$ is isotopic to the identity in $MCG(\S).$
\end{theorem}

The content of this paper is described as follows.
In Section 2, we review intersection Floer homology
for unobstructed immersed curves
 on a Riemann surface of genus greater than one. Here we introduce the admissibility condition 
 and   show isotopy invariance of Floer homology.
  The $A_{\infty}$-structure and the Fukaya category in this setup are defined in Section 3. For the general treatment one can see
 \cite{fukaya2001floer, seidel2008fukaya}. In Section 4, we give an explicit construction of a quasi-isomorphism between two $C^1$-close curves, which is used to show that any two isotopic curves are quasi-isomorphic objects. These quasi-isomorphisms serve as weak identity morphisms in the sense of  \cite{alex2009proof}. The Grothendieck group of the derived Fukaya category is computed in Section 5. Following Abouzaid, we derive the Fukaya category using twisted complexes and identify the zero objects. Next we construct the map from $K_0(Fuk(\S))$ to $H_1(S\S;\mathbb{Z})$ and use well-known results about mapping class groups to prove injectivity. In the last section we define an action of the mapping class group on this category and prove that this action is faithful using the Alexander's trick. 

In the Appendix we discuss Maslov index and show that the combinatorial counting of polygons coincides with that in Floer theory.

\subsection*{Acknowledgments}
The authors are grateful to Mohammed Abouzaid and Paolo Ghiggini for helpful discussions. This research was supported by the LUMS faculty research startup grant awarded to the first author.  The second author would like to thank the Abdus Salam School of Mathematical Sciences, Lahore for their invitation.

\section{Reviewing Floer Homology}\label{Floer}
 Basic objects in this paper are {\em unobstructed curves}, which means
 oriented immersed closed curves whose lifting to the universal cover is a properly embedded line.
 A closed connected curve is unobstructed if and only if it is an immersion which is neither null-homotopic and nor does it bound an immersed fishtail \cite[Lemma 2.2]{abouzaid2008fukaya}.

In this section we review the definition of Floer homology
for unobstructed curves.
 Floer chain complex will be generated by intersection points, which requires the usual transversality
 condition on the pair of curves giving these intersection points. Defined in this naive manner our chain complexes are not invariant under isotopy. For a given choice of disjoint representatives of two curves the intersection could be empty, whereas a little perturbation may give rise to intersection points. To overcome this we introduce the following \emph{admissibility condition}. This forbids us from considering a pair of curves bounding a cylinder.
By a $2$-chain we mean a linear combination of the connected components of the complement of curves. Any map from a surface with boundary $(S,\partial S)$ to $(\S,\gamma_1\cup\gamma_2)$ defines a $2$-chain. 
An Euler zero $2$-chain is one which is defined with a surface $S$ whose Euler characteristic is zero.

\begin{defn}\label{admissibility} A pair of  unobstructed curves $(\g_1,\g_2)$ is admissible, if any
 Euler zero 2-chain between them has regions with positive and negative coefficients.
In particular, two non-homologous curves always form an admissible pair.
\end{defn}

All curves as objects of our category are going to be equipped with a  {\em marked point}.
Marked points are going to be useful later in the context of smoothing of intersection points between curves (see Section \ref{cones}).
 We define below the Floer chain complex for an admissible pair of  unobstructed  curves with marked point $(\g_1,\g_2)$.

 We will shorten and say {\em admissible pair} for admissible pair of  unobstructed  curves with marked point and assume transverse intersection.
\begin{defn}
The Floer chain complex for the admissible pair 
$(\g_1,\g_2)$ is the $\Z/2\Z$-graded module over $\mathbb{Z}$ generated by intersection points $p$ between $\g_1$ and $\g_2$, with even degree (resp. odd degree)
 if  the intersection sign $\epsilon_p$ is negative (resp. positive).
  Namely
  $$CF^i(\g_1,\g_2)=\mathop{\mathop{\sum}\limits_{p\in\g_1\cap\g_2}}\limits_{\epsilon_p=(-1)^{i+1}}\mathbb{Z}\,p$$
\end{defn}

For points $p,q\in \g_1\cap \g_2$, a {\em bigon} $u$ from $p $ to $q$ is, up to a smooth reparametrization,  an oriented immersion of the half disc 
$ \mathbb{D}=\{z\in \mathbb{C}\ :\ |z|=1 \text{ and } \Im(z)\geq 0\}$, $$u:\mathbb{D}\to \S$$ which is mapped as follows:\\
the point $1,-1\in \mathbb{D}$ are mapped to the points $p,q \in \partial u$ respectively. The clockwise and anti-clockwise paths from $p$ to $q$ are mapped to subsets of $\g_1$ and $\g_2$ respectively (see Figure \ref{bigon}).
Note that our immersed bigons are equivalent to what the authors call smooth lunes in \cite{de2014combinatorial}.
\begin{figure}
\begin{picture}(0,110)
 \put(-95,0){\def\svgwidth{0.45\textwidth}
 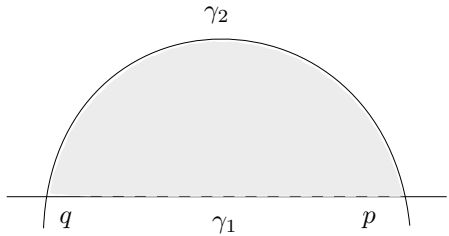} 
\end{picture}
\caption{A bigon from $p$ to $q$}
\label{bigon}
\end{figure}

By the Riemann mapping theorem, each bigon  admits an holomorphic representative which is unique up to
reparameterization by automorphisms of $(\mathbb{D};1,-1)$.

We show in Appendix \ref{appendix_bigons} that the boundary  map defined below  coincides with the counting of rigid holomorphic discs in lagrangian Floer theory.

\begin{defn}
Let $p,q\in \g_1\cap \g_2$, we define $\mc{M}(q,p) $ to be the set of bigons from $p$ to $q$.
\end{defn}
 One can choose a \emph{local model near $q$} by sending $\g_1$ to the $x$-axis and $\g_2$ to the $y$-axis. The intersection of the bigon with a neighborhood of -1 lies either in the first quadrant $\{x>0,y>0\}$ or the third quadrant $\{x<0,y<0\}$.
The \emph{differential} \mbox{$d: CF^i(\g_1,\g_2)\to CF^{i+1}(\g_1,\g_2)$} is defined by counting immersed bigons.

\begin{defn} We define the differential by the following formula $$d(p)=\mathop{\sum}\limits_{q\in u\in \mc{M}(q,p)} (-1)^{s(u)}q,$$
where $s(u)$ is the sum of a contribution of 1 for each marked point on $\partial u$ and a contribution of 1 whenever the given orientation of $\g_2$ disagrees with its orientation as the boundary of $u$.
\end{defn}
Note that the sign is fixed so that the boundary map in Floer complex has degree $+1$.
The admissibility condition ensures finitenes of the sum. We give a proof for polygons in Section \ref{infinity}.
The following lemma shows that $(CF^*(\g_1,\g_2),d)$ is indeed a chain complex.

\begin{lem}
The Floer differential satisfies $d^2=0.$
\end{lem}
\begin{proof}It is enough to show that for any pair of bigons $(u,v)$: $u\in \mc{M}(q,p)$, $v\in \mc{M}(r,q)$, there exists a non-convex corner point $q'$ lying on $\g_1$ (see Figure \ref{fig:differential}). This point gives rise to a new pair of bigons $(u',v')$ such that the signed contributions cancel pairwise. Since the curves we consider are unobstructed immersed curves, the tangent vector to $\g_2$ at $r$ points upward thus the curve only intersects $\g_1$ and not itself. Let the resulting intersection point with $\g_1$ be denoted by $q'$, then $q'$ gives rise to bigons $u'\in \mc{M}(q',p)$ and $v'\in \mc{M}(r,q')$. The parity of signed contributions $s(u)+s(v)$ and $s(u')+s(v')$ differs and so they cancel pairwise giving
 $$d^2p=d((-1)^{s(u)}q+(-1)^{s(u')}q')=(-1)^{s(u)+s(v)}r+(-1)^{s(u')+s(v')}r=0.$$
\end{proof}

\begin{figure}
\begin{picture}(0,110)
 \put(-135,0){\def\svgwidth{0.35\textwidth}
 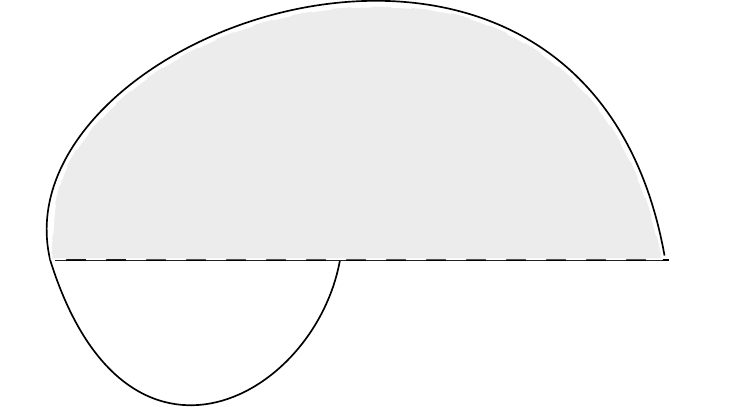} 
\put(10,0){\def\svgwidth{0.35\textwidth}
 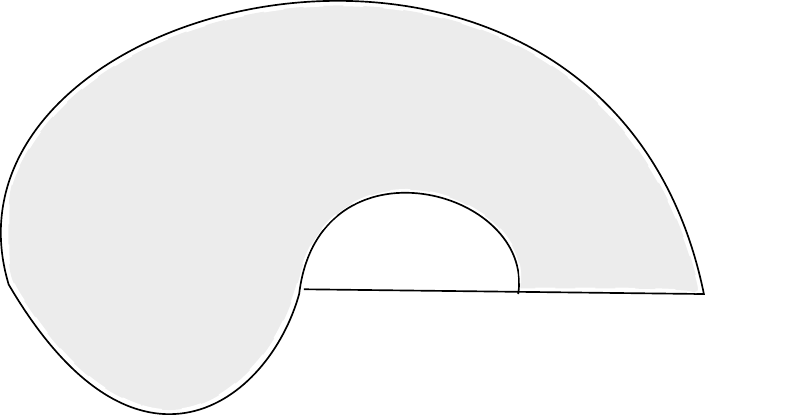} 
\end{picture}

\caption{Bigons whose contributions cancel pair-wise}
\label{fig:differential}
\end{figure}

\begin{defn} Given an  admissible pair of curves $(\g_1,\g_2)$, Floer cohomology is defined as the cohomology of the chain complex $(CF^*(\g_1,\g_2),d)$ and is denoted by $$HF^*(\g_1,\g_2).$$
\end{defn}

Following Bouchair \cite{bouchair2009cat} we can prove the following invariance theorem. A stronger result involving quasi-isomorphisms is given in Section \ref{quasi}.
\begin{theorem}[Isotopy invariance]\label{E:isotopy}
Let $\g_1,\g_2$, be admissible curves.  If  $\g'_1,\g'_2$  are admissible and respectively isotopic to $\g_1,\g_2$, then
$HF^*(\g_1,\g_2)$ and $HF^*(\g'_1,\g'_2)$ are isomorphic.
\end{theorem}
\begin{proof}
Using \cite[Section 3.3]{fathi1991travaux}, it is enough to consider an elementary isotopy of one of the curves which create two intersection points $a$ and $b$
connected by a bigon from $a$ to $b$.
Let us consider the case where $\g_2'=\g_2$, and $\g_1'\cap \g_2=(\g_1\cap \g_2)\cup\{a,b\}$.
Inclusion of intersection points identifies $CF(\g_1,\g_2)$ with a submodule  $\mathcal{C}\subset CF(\g_1',\g_2)$.
The inclusion is not a chain map. However there is only one bigon from $a$ to $b$, and following Bouchair argument \cite[Theorem 3.3]{bouchair2009cat} (see also \cite{bar2007fast}  ) we may apply to the complex $\mathcal{C}'$ a Gaussian elimination which produces a subcomplex $\mathcal{C}''\subset \mathcal{C}'$ which is homotopy equivalent to $\mathcal{C}'$ and  isomorphic to $\mathcal{C}$.
\end{proof}

\section{The $A_{\i}$-structure}\label{infinity}
 The $A_{\i}$-structure constitutes higher order compositions of morphisms. Each composition is only associative upto homotopies forming a sequence of obstructions for associativity of some higher product. For details one can see \cite{fukaya1993morse,seidel2008fukaya}.
 
 The definition of the Fukaya category requires a coherent extension of the admissibility condition.
\begin{defn}\label{admis_polygon}Admissibility: A finite set of curves is admissible if any integral Euler zero  $2$-chain representing a homology between some of these curves has positive and negative coefficients. Here the $2$-chain is represented by a linear combination of the connected components of the complement of the curves and Euler zero means that it can be defined by a map from a surface whose Euler characreristic is zero.
\end{defn}

Let us consider an admissible collection of unobstructed  transverse curves $\{\g_i\}_{i=0}^k$ without having triple intersections.
A {\em $(k+1)$-gon} is an orientation preserving map from the disc $D$ to $\S$, with $k+1$ marked points $a_0$,\dots,$a_k$ on the boundary such that the consecutive arcs
$\wideparen{a_ia_{i+1}}$  on  $\partial D$ are mapped to the respective curves $\g_i$.

The higher product $m_k$ is the following map
$$m_k:CF^*(\g_{k-1},\g_{k})\o\ldots \o CF^*(\g_{1},\g_{2})\o CF^*(\g_{0},\g_{1})\to CF^*(\g_{0},\g_{k}) $$  defined by counting oriented immersed $(k+1)$-gons
with convex corners.
We denote by $p_{i,i+1}$ an intersection point between curves $\g_i$ and $\g_{i+1}$.

One can choose a \emph{local model near $p_{i,i+1}$} by sending $\g_i$ to the $x-axis$ and $\g_{i+1}$ to the $y-axis$. The polygon thus lies either in the second quadrant $\{x<0,y>0\}$ or the fourth quadrant $\{x>0,y<0\}$. Near $p_{0,k}$ we have to choose the opposite local model by sending $\g_k$ to the $x-axis$ and $\g_{0}$ to the $y-axis$ for the polygon to lie in even quadrants.
\begin{defn}\label{polygons}
Let $\mc{M}(p_{0,k},p_{k-1,k},\ldots,p_{0,1})$ denote the set of $(k+1)$-gons $u$ such that the boundary $\partial u$ is a union of segments $\overline{p_{i-1,i},p_{i,i+1}}\subset \g_i$,
the corners are convex and the natural orientation of the boundary of $u$ corresponds to increasing $i$.
\end{defn}

Here $(k+1)$-gons are considered modulo an action of $Aut(D)$ by reparameterization. By the Riemann mapping theorem, one can fix the positions of three marked points and the resulting moduli space of conformal structures on $D$ is $k-2$ dimensional. For a fixed conformal structure and a homotopy class $[u]$, assuming transversality holds, the expected dimension of the moduli space is $$\dim{\mc{M}(p_{0,k},p_{k-1,k},\ldots,p_{0,1})}=k-2+ Ind([u]).$$
Here $Ind([u])$ is the index of the Cauchy-Riemann operator with appropriate boundary conditions. 
Note that $(k+1)$-gons need not be embedded and finiteness is not obvious.
 The admissibility condition ensures finiteness as shown below. 
 
The Euler measure of a surface $S$ with corners on the boundary is $\frac{1}{2\pi}$ times the
integral over $S$ of the curvature of a metric on $S$ for which the boundary  is geodesic and the corners  are right angles. We use the notation $e(S)$. An important property
is additivity for gluing along segments of the boundary. In the case of smooth boundary it coincides with the Euler characteristic. For a $n$-gon $D$ with convex corners, the value is $e(D)=1-\frac{n}{4}$.

Let $C:=C_2(\S;\mathbb{Z})$ denote the free module over $\mathbb{Z}$ with basis consisting of the connected components of $\S\setminus (\cup_i \g_i)$. The Euler measure $e$ extends to a linear form on $C$. Let $H\subset C$ be the group of $2$-chains in the kernel of $e$ whose boundary is represented by a linear combination of the curves $\g_i$. 
We will denote by $H_{\mathbb{F}}$ and $C_2(\S; \mathbb{F})$ the $\mathbb{F}$-span of $H$ and $C$ respectively, where $\mathbb{F}$ is a field. We let the positive cone in $C$ and $C_2(\S,\mathbb{R})$ be denoted by $C_2^+(\S;\mathbb{Z})$ and $C_2^+(\S;\mathbb{R})$ respectively. Recall that a collection of curves $\Gamma=\cup_i \g_i$ is admissible if and only if any non-zero $b\in H$ has positive and negative coordinates in the canonical basis of $C$ over $\mathbb{Z}$. Then the property will also holds for any $b\in H_\mathbb{Q}$.

\begin{lem}
If  $\Gamma$ is admissible then any non-zero $b\in H_{\mathbb{R}}$ has positive and negative coordinates in the canonical basis of $C$.
\end{lem}
\begin{proof}
Suppose $b\in H_{\mathbb{R}}\cap C_2^+(\S;\mathbb{R})$ is non-zero. Let $J=\{i| b_i=0 \}$. Note that $ H_{\mathbb{R}}\cap \{x_i=0| i\in J\}$ is the $\mathbb{R}$-span of $ H\cap \{x_i=0| i\in J\}$, hence
 there exists a sequence $b_n\in H_{\mathbb{Q}}\cap \{x_i=0| i\in J\}$ such that $\mathop{\lim}\limits_{n\to \infty} b_n=b$. The non-zero components of $b$ are all positive, thus corresponding components of $b_n$ are all positive for $n$ big enough. This contradicts the fact that $b_n$ have positive and negative coordinates.
\end{proof}

\begin{prop}\label{compact}
If $\Gamma$ is admissible and $u\in C_2^+(\S; \mathbb{R})$, then $(u+H_{\mathbb{R}})\cap C_2^+(\S; \mathbb{R})$ is compact.
\end{prop}
\begin{proof}
For any $b\in H_{\mathbb{R}}$, let $\mathbb{R}b$ denote the $\mathbb{R}$-span of $b$. Then the intersection of the affine line $(u+\mathbb{R}b )$ with the positive cone $ C_2^+(\S; \mathbb{R})$ is  compact. 
We put the natural Euclidean structure on $C_2(\S;\mathbb{R})$ and
define a map $$d: S(H_{\mathbb{R}})\to \mathbb{R}^+,$$ where $ S(H_{\mathbb{R}})$ is the unit sphere in $H_{\mathbb{R}}$, by sending each $b$ to the diameter of the segment $\mathbb{R}b\cap C_2^+(\S; \mathbb{R})$. This map is continuous on a compact set, and for \mbox{$R=\max\{d(b)|  b\in S(H_{\mathbb{R}})\}$} we have $$(u+H_{\mathbb{R}})\cap C_2^+(\S; \mathbb{R})\subset B(u,R),$$
where $B(u,R)$ is a ball centered at $u$ with radius $R$. Thus we have a closed set in a Euclidean space which is also bounded.\end{proof}

\begin{cor}
Given an admissible  collection of unobstructed transverse curves $\Gamma=\cup_{i=0}^k\g_i$, the  set of polygons $\mc{M}(p_{0,k},p_{k-1,k},\ldots,p_{0,1})$ as in Definition \ref{polygons} is finite.
\end{cor}
\begin{proof}
If the set of polygons $\mc{M}(p_{0,k},p_{k-1,k},\ldots,p_{0,1})$ is non-empty, we choose \\$u\in 
\mc{M}(p_{0,k},p_{k-1,k},\ldots,p_{0,1})\subset C_2^+(\S;\mathbb{Z})$.
Then for any $v\in 
\mc{M}(p_{0,k},p_{k-1,k},\ldots,p_{0,1})$, we have $v=u+h$ with $h\in H$, indeed $u$ and $v$ have equal Euler measure.
From Proposition \ref{compact}, $(u+H )\cap C_2^+(\S;\mathbb{Z})$ is finite which proves the corollary.
\end{proof}
Given an admissible  collection of unobstructed transverse curves $\Gamma=\cup_{i=0}^k\g_i$, it follows from the above corollary that the  set of polygons $\mc{M}(p_{0,k},p_{k-1,k},\ldots,p_{0,1})$ is finite.
\begin{rem}
If we consider an immersed square $u$ and a cylinder $v$ bounded by two homologous curves, whose bounding curves are coincident with two opposite sides of $u$ (see Figure \ref{finiteness}). Then the 2-chain $u+kv$ for any integer $k$, is also an immersed square which contributes to the product $m_3$. This highlights the need for admissibility, which ensures finiteness, so that the higher products are well defined.
\end{rem}
\begin{figure}[ht]
\def\svgwidth{0.45\textwidth}
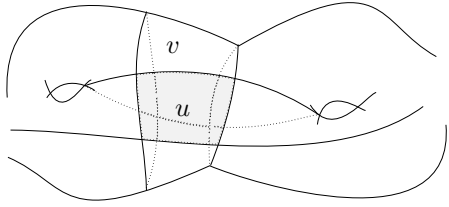
\label{finiteness}
\caption{The 2-chain $u+kv.$}
\end{figure}

\begin{defn}The higher product
$$m_k:CF^*(\g_{k-1},\g_{k})\o\ldots \o CF^*(\g_{1},\g_{2})\o CF^*(\g_{0},\g_{1})\to CF^*(\g_{0},\g_{k}) $$
is defined by the formula: $$m_k(p_{k-1,k},\ldots,p_{0,1})=\mathop{\sum}\limits_{p_{0,k}\in u\in \mc{M}(p_{0,k},p_{k-1,k},\ldots,p_{0,1})}(-1)^{s(u)}p_{0,k},$$
where $s(u)$ is the sum of a contribution of 1 for each marked point on $\partial u$ and a contribution of 1 for each odd degree intersection point $p_{i,j}$, for which the given orientation of $\g_j$ disagrees with its orientation as the boundary of $u$. For instance, $m_1$ is just the differential $d$.
\end{defn}

\begin{lem} The product $m_2:CF^*(\g_{1},\g_{2})\o CF^*(\g_{0},\g_{1})\to CF^*(\g_{0},\g_{2}) $ is a chain map.
\end{lem}
The product $m_2$ is only associative upto homotopy and so a homotopy can be constructed, which is just $m_3$, due to the $A_{\infty}$-relation $\pm m_2(m_2\o 1)\pm m_2(1\o m_2)= d\circ m_3 \pm m_3(d\o 1\o 1)\pm m_3(1\o d\o 1)\pm m_3(1\o 1\o d).$
\begin{lem}
The higher products $m_k$ satisfy the following $A_{\infty}$-equation
$$\mathop{\sum}\limits_{l,s,t}(-1)^{\maltese_t}m_l(p_{k-1,k},\ldots, m_t(p_{s+t,s+t+1},\ldots,p_{s,s+1}),\ldots,p_{0,1})=0$$
where $\maltese_t =t+\mathop{\sum}\limits_{i=1}^t \deg{p_{i-1,i}}.$
\end{lem}
The proof relies on an analysis of the one dimensional moduli spaces of rigid polygons
(see Appendix B for details) and their compactifications. For marked points $z_0,\ldots, z_{k-1},w\in \partial D$, fix the images to be the corners $p_{0,1},\ldots, p_{k-1,k},p_{0,k}$ for representatives of a fixed homotopy class $[u]$ with $Ind([u])=2-k$.

For instance, when $k=3$, $$\pm m_2(m_2(\cdot,\cdot),\cdot)\pm m_2(\cdot, m_2(\cdot,\cdot))= d\circ m_3 \pm m_3(d(\cdot),\cdot,\cdot)\pm m_3(\cdot, d(\cdot),\cdot)\pm m_3(\cdot, \cdot,d(\cdot))$$ corresponds to the degenerate configurations of discs which form the boundary of the compactification of the moduli space in question.
In general for $k>3$, the moduli space compactifies to a one dimensional manifold whose boundary points correspond either to a Maslov index 1 disc degenerating at one of the $k+1$ marked points or a pair of discs with each component having at least two marked points. The first case of degeneration corresponds to terms in $A_{\infty}$-equation involving $m_1$, while the second case corresponds to the remaining terms. For $k=3$, the configurations with a strip of Maslov index 1 together with an index -1 disc with 4 marked points gives the right hand side. Whereas those consisting of a pair of index 0 discs with at least two marked points account for the two terms on the left hand side. We give a combinatorial proof for $k=3$.

\begin{figure}
\begin{picture}(0,150)
 \put(-135,10){\def\svgwidth{0.4\textwidth}
 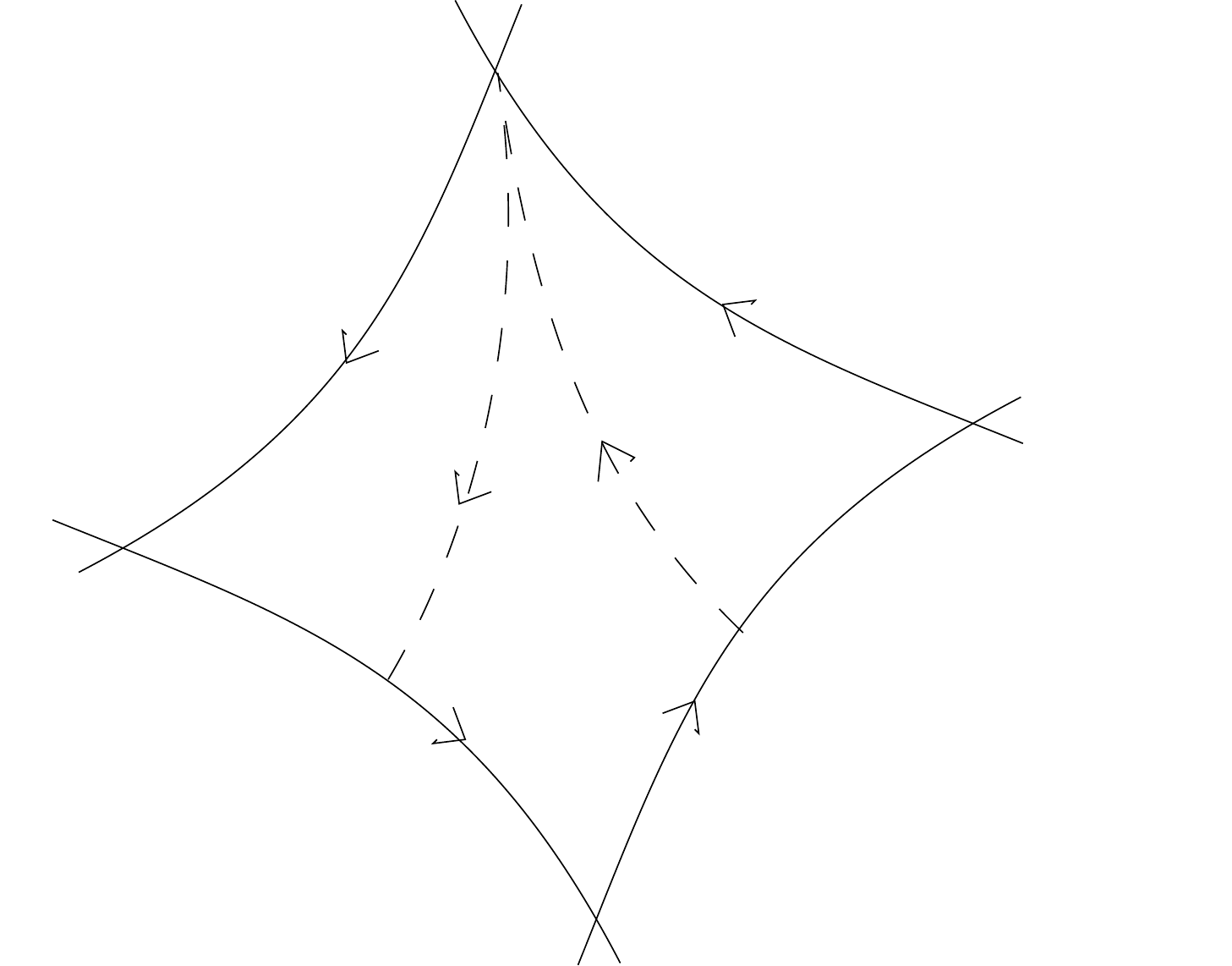} 
 \put(60,55){\def\svgwidth{0.16\textwidth}
 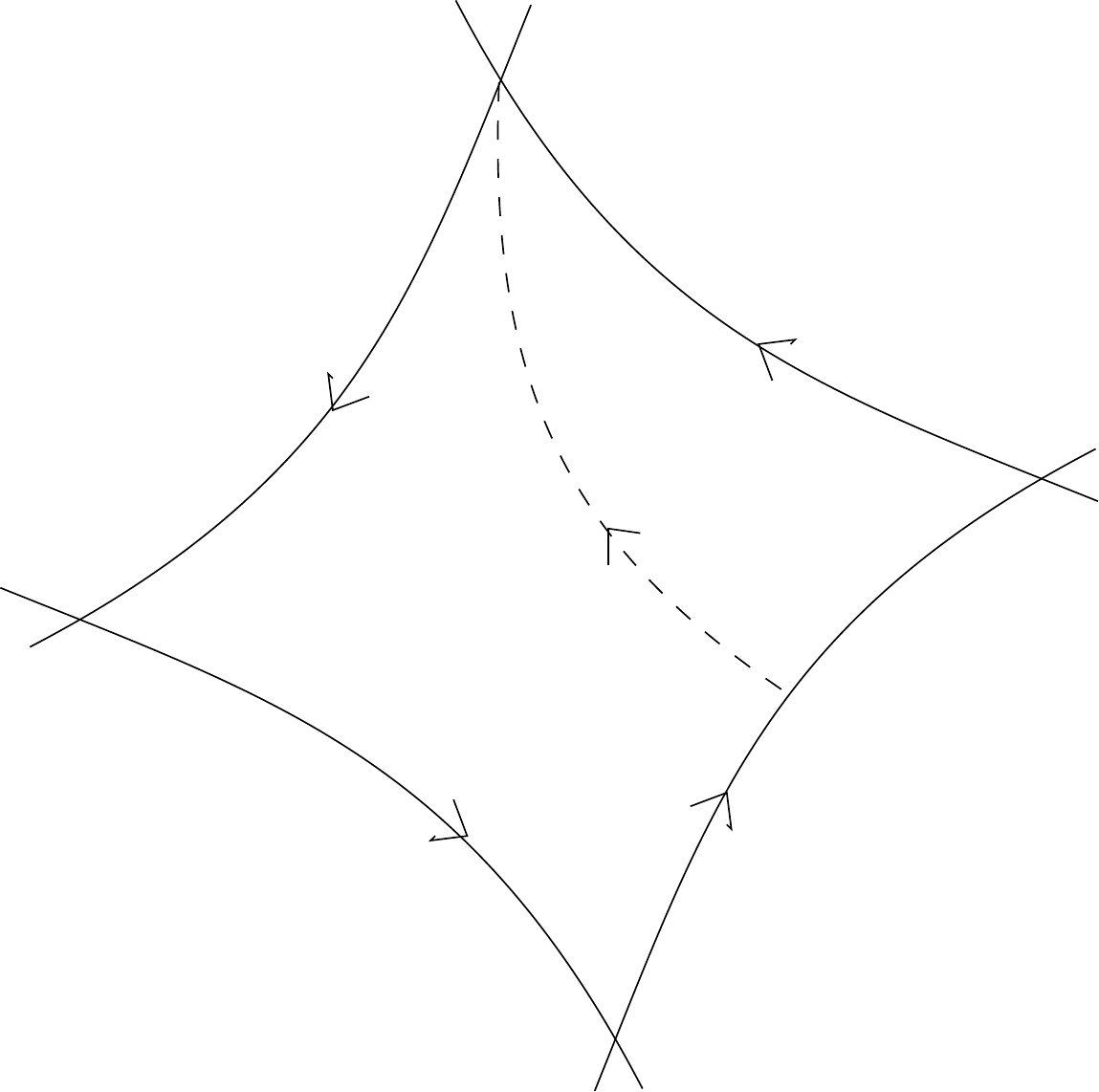} 

 \put(20,15){\def\svgwidth{0.16\textwidth}
 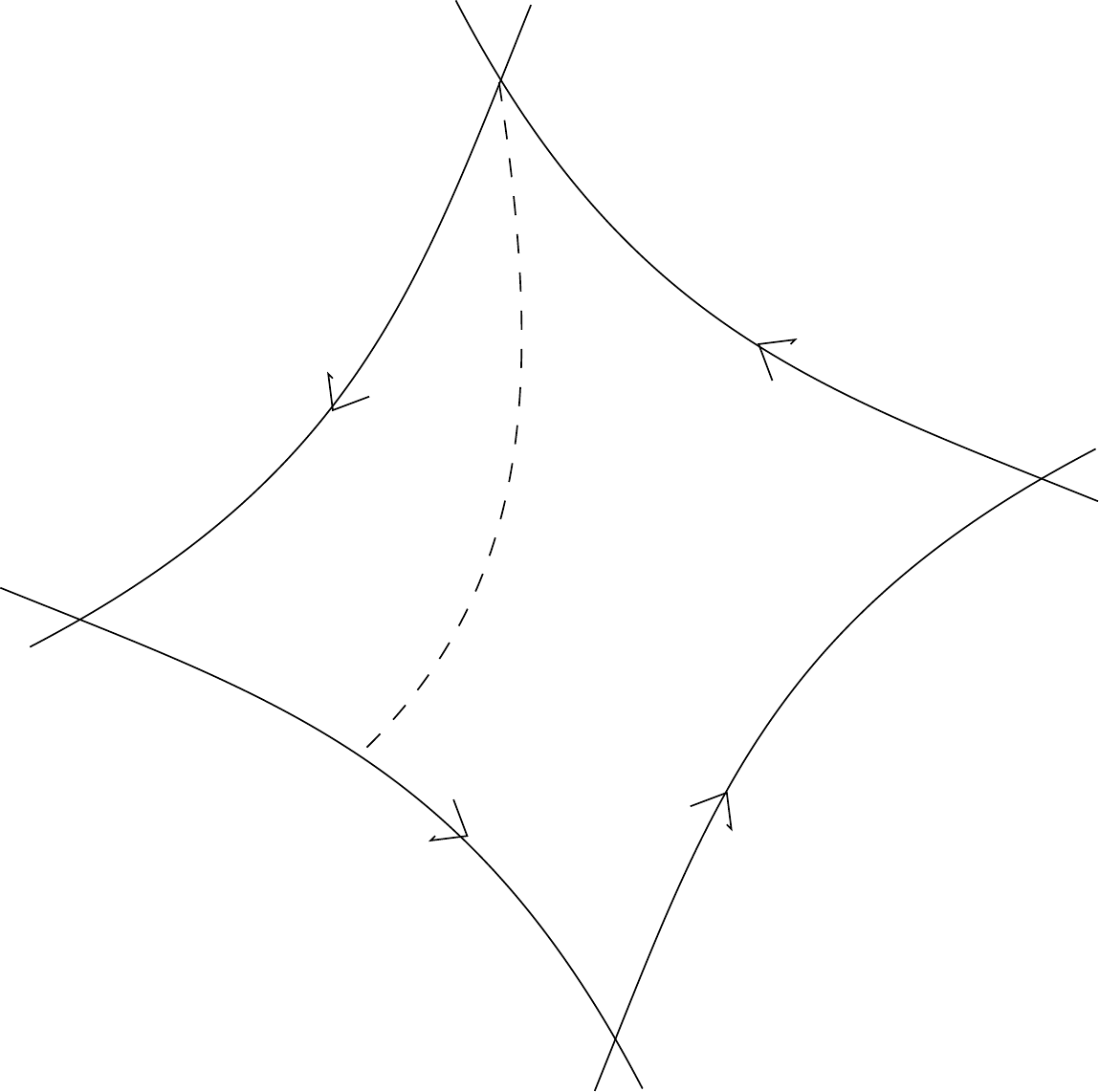} 
\end{picture}

\caption{Case of a square when $k=3.$}
\label{fig:k=3}
\end{figure}

\begin{proof}
By definition of $m_3$, one must show that the contributions of signs for each pair of decompositions $(u,v)$, $(u',v')$ for a given polygon $w$ cancel pairwise. That is, their total signs differ in parity:
$$\maltese_0+s(u)+s(v)+\maltese_{1}+s(u')+s(v')\equiv 1 \,mod\,2.$$
Let there be two points on the boundary of the 4-gon $w$: $p_{0,2},p_{1,3}$ (see Figure \ref{fig:k=3}). These points give two decompositions of $w$ into a pair of polygons $(u,v)$ and $(u',v')$. Since the contributions of all vertices except $p_{0,2},p_{1,3},p_{1,2}$ are the same for $s(u)+s(v)$ and $s(u')+s(v')$, thus we can safely ignore them. Also since all the marked points contribute the same to $s(u)+s(v)$ and $s(u')+s(v')$, they can be ignored as well. The corner $p_{1,2}$ contributes a sign to each $s(u')+s(v')$ and $s(u)+s(v)$, while $p_{0,2}$ contributes a sign to $s(u)+s(v)$. The contribution of $p_{1,3}$ is 1 to $s(u')+s(v')$ and so we have
$$\maltese_0+2+\maltese_{1}+2\equiv 1 (mod2) \Leftrightarrow 1+\deg{p_{0,1}}\equiv 1\,mod\,2.$$
This proves the $A_{\infty}$-relation for $k=3$.
\end{proof}

	Since a curve cannot be transverse to itself, there is a difficulty in defining a notion analogous to that of identity morphism for $A_{\infty}$-categories. Following Abouzaid we also adopt the Kontsevich-Soibelman framework \cite{kontsevich2001homological} of $A_{\infty}$-precategories.
For a detailed treatment of the subject the reader may refer to {\em A proof of the Kontsevich-Soibelman conjecture} by Efimov in \cite{alex2009proof}. 
According to Efimov's terminology, the following defines a non-unital $A_{\infty}$-precategory


(i) The objects are unobstructed immersed curves in $\S$ with a marked point.

(ii) The admissible transverse subsets $\mc{O}b(Fuk(\S))^{tr}_n$ are given by $n$-tuples of admissible curves such that no three curves have a common point.

(iii) Morphisms are Floer chain complexes $CF(\g_1,\g_2)$ generated by intersection points between curves $\g_1$ and $\g_2$. Among the transversal objects, the higher compositions are given by $m_k$.


\noindent The Fukaya category $Fuk(\S)$ of a surface $\S$ is defined by the above data. 

In any $A_{\infty}$-category $\mathcal{C}$, an $m_1$-cohomology class $[e]$ of a morphism $e:A\to A$ is called a `weak identity morphism', if for any closed morphisms $i:A\to B$ and $j:C \to A$, $[e]\cdot [i]=[i]$ and $[j]\cdot [e]=[j]$, for any objects $B,C$ (see \cite[Definition 2.7]{alex2009proof}). In the following section we construct quasi-isomorphisms that serve as weak identity morphisms making the above an $A_{\infty}$-precategory.
\section{weak identity morphisms}\label{quasi}
Assuming the curves we consider are transverse, it is not possible to speak of endomorphisms of objects $CF(\g,\g)$, but one could still speak of quasi-isomorphisms of objects. Although isotopy invariance has been discussed in Theorem \ref{E:isotopy}, in this section we construct an explicit quasi-isomorphism between two $C^1$-close curves and use it to find a sequence of quasi-isomorphisms between any two isotopic curves. 
 These canonical quasi-isomorphisms between $C^1$-close curves play the role of weak identity morphisms providing the extension property required for an $A_{\infty}$-precategory.

We begin with reminding the following definition of quasi-isomorphism between objects of an $A_{\infty}$-precategory in general. 

\begin{defn}
An element $f\in Mor(A,A')$ satisfying $m_1(f)=0$ is called a quasi-isomorphism if the maps $$\begin{array}{cc}
                                                                                                m_2(\cdot,f): & Mor(A',B')\to Mor(A,B') \\
                                                                                                m_2(f,\cdot): & Mor(B,A) \to Mor(B,A')
                                                                                              \end{array}
$$ are quasi-isomorphisms of complexes, whenever $(B,A,A')$ or $(A,A',B')$ are triples of transverse objects.
\end{defn}
We consider the equivalence relation generated by existence of a quasi-isomorphism.
\begin{defn}
Two objects $A$ and $B$ are quasi-isomorphic if and only if there exists a sequence of objects $A=A_0$, $A_1$, \dots , $A_m=B$,
and quasi-isomorphisms $f_i$ either from $A_{i-1}$ to $A_i$, or from $A_i$ to $A_{i-1}$, $0<i\leq m$.
\end{defn}
An $A_{\infty}$-precategory is a non-unital $A_{\infty}$-precategory which satisfies the following extension property (see  \cite[Definition 2.15]{alex2009proof}). For any set of transversal sequences $\Gamma_{\alpha}$ of closed immersed curves and a given curve $\g$, one can always find objects $\g_-, \g_+$ and quasi-isomorphisms $f_-:\g_-\to\g$ and $f_+:\g\to\g_+$ such that all sequences $(\g_-,\Gamma_{\alpha},\g_+)$, $\alpha\in I$ are transversal. These quasi-isomorphisms serve as weak identity morphisms in this $A_{\infty}$-precategory and these are what we want to construct in this section. 
 
We say that a curve $\b$ is \emph{jointly transverse} to $C^1$-close curves $\g_1$ and $\g_1 '$, if $\b$ is tranverse to both curves and $\g_1$, $\g_1'$ are $C^1$-close with respect to $\b$. In particular, there is a canonical bijection between the sets $\b\cap\g_1$ and $\b\cap\g_1'$.
\begin{prop}\label{jointly_prop}
Consider a curve $\g_1$ and let $\g_1'$ be $C^1$-close to $\g_1$, such that $|\g_1\cap\g_1'|=2$. Let the degree zero intersection point be denoted by $p$. Then,

(i) $p$ is closed as a morphism between the curves $\g_1$ and $\g_1'$,

(ii) for any curve $\b$ jointly transverse to $\g_1$ and $\g_1'$, the map $$m_2(p, \cdot):CF(\b,\g_1)\to CF(\b,\g_1')$$ is an isomorphism of complexes.
\end{prop}
\begin{proof}
Let $q$ denote the odd degree intersection point. The curves $\g_1$ and $\g_1'$ are closed, hence $m_1(p)=-q+q=0$. The second implication follows by definition of jointly transverse.  \end{proof}

In the rest of the section we prove two statements. First we generalise the second assertion in the previous proposition, to the case of any two $C^1$-close curves (see Proposition \ref{C^1}) and then we show that for any two isotopic curves there is an isomorphism at the level of homology (see Theorem \ref{isotopic}).
To this end, we prove that there are no non-trivial higher products in our setting, and from the $A_{\infty} $ equation the associativity of $m_2$ follows.
\begin{lem}
Consider $C^1$-close curves $\g_1$ and $\g_1'$.  Given a finite collection of mutually transverse curves $\{\a_j\}\cup \g_1$, the moduli space of polygons (4-gons and higher) through $p$ is empty.
\end{lem}
\begin{proof}
Let the curves $\g_1$ and $\g_1'$ intersect in points $p$ and $q$, we denote by $v$ the bigon from $p$ to $q$.
For such a given finite collection of mutually transverse curves $\{\a_j\}\cup\g_1$ (here $\g_1$ and $\g_1'$ are $C^1$-close with respect to $\{\a_j\}$), the above assertion amounts to showing that $$m_k(x_k,\ldots,x_{d+2},p,x_{d},\ldots,x_1)=0$$ whenever
$x_{i}\in \a_{n_i}\cap \a_{n_{i+1}}$ and $k>2$.
Let $u$ be an element of the corresponding moduli space, which agrees with $v$ near $p$. Being $C^1$-close to $\g_1$, the curve $\g_1'$ intersects $\a_{n_d}$ at a point $x_d'$ splitting $v$ into two triangles.

Let us denote by $u_1$ the triangle with vertices $x_d, x_d' $ and $p$. The polygon $u$ is a map $D\to \S$. Consider the region $D'$ in $D$ corresponding to the inverse image of $u_1$ under $u$, containing the marked point $z_p\in \partial D$ that maps to $p.$ By our assumption there are two points $z_1$ and $z_2$ on $\partial D'$ and segments from $z_p$ to $z_i$ along $\partial D$ which map to $\g_1\cup \a_{n_{d}}$ and $\g_1'$ respectively.
The map $u$ has critical points only at inverse images of  $p,x_d,x_{d+2}$ and at marked points that are inverse images of intersection points between $\a_{n_i}$ and $\a_d$. Since $\g_1$ and $\g_1' $ are $C^1$-close, no marked points lie in $v$, hence no other critical points lie on the segment connecting $z_1$ and $z_2$ in $\partial D'$. The map $u$ is thus an immersion on the remaining segment, but this means that $z_1$ and $z_2$ are in fact the same point. This proves there are no more than three marked points, thus no 4-gons (or higher polygons) either. Therefore $m_{\geq3}=0$.
\end{proof}

Let $\{u_k\}_{k=0}^{2n-1}$ denote bigons with disjoint interiors bounded by isotopic $C^1$-close curves $\g_1$ and $\g_{1}'$. Further let each of these bigons lie between several even and odd degree intersection points $\{p_i\}_{i=0}^{n-1}$ and $\{q_i\}_{i=0}^{n-1}$ respectively.
\begin{figure}[ht]
\def\svgwidth{0.55\textwidth}
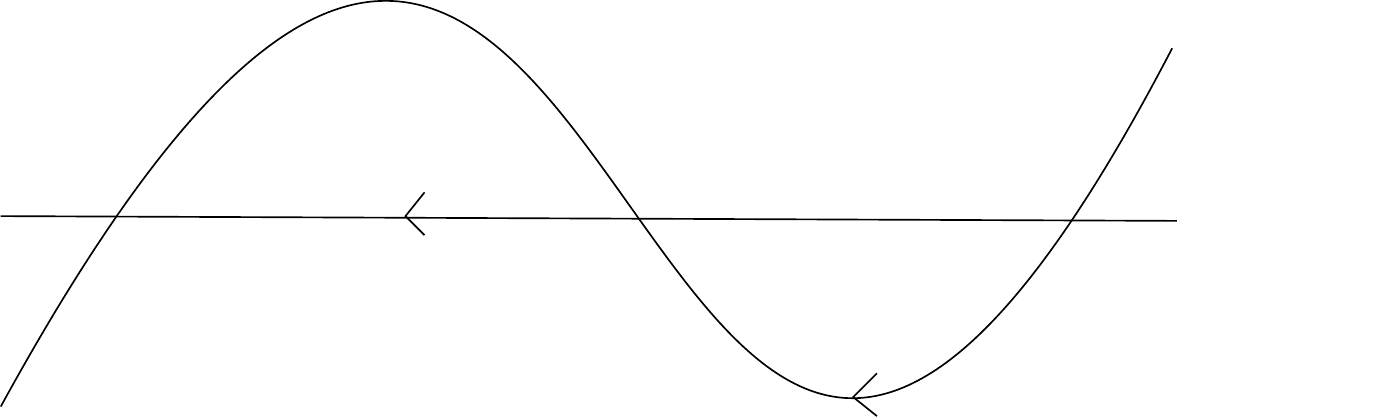
\caption{}
\label{closed}
\end{figure}

For each $p_i$ (see Figure \ref{closed}) one can write $$d(p_i)=-q_{i-1}+q_i.$$  We define $$f_{1,1'}=\mathop{\sum}\limits_{i=0}^{n-1}p_i$$ and obtain $d(f_{1,1'})=m_1(f_{1,1'})=0$, which follows from the fact that the curves are closed. 
Here as before we implicitly denote by $f_{i,j}$ the morphism between objects $\g_i$ and $\g_j$. 

\begin{lem}
Let $\g_1$ and $\g_2$ be isotopic closed curves bounding immersed discs with disjoint interiors, then 
$$m_2(f_{2,1'},f_{1,2})=f_{1,1'}.$$

\end{lem}

\begin{figure}[ht]
\def\svgwidth{0.65\textwidth}
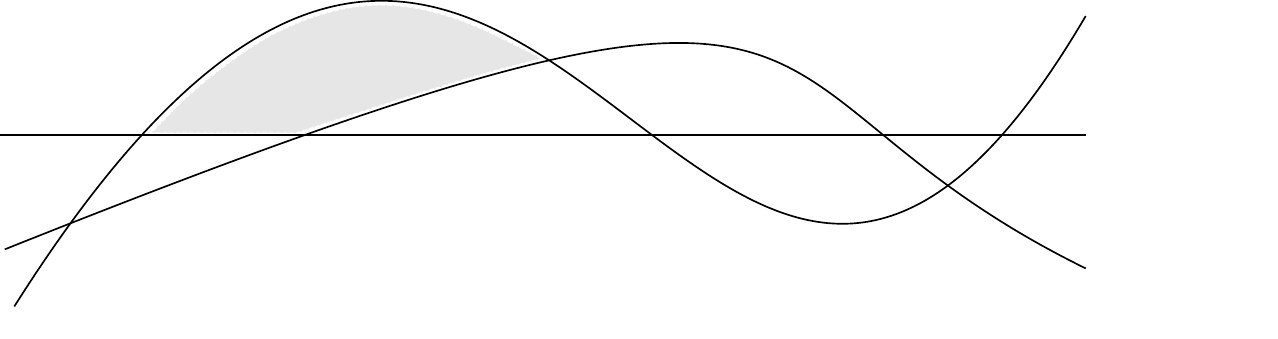
\caption{}
\label{fig:figure_5_v7}
\end{figure}

\begin{proof}We denote the bigons bounded by $\g_1\cup\g_2$ with corners $p_i$ and $q_i$ by $u_k$ as in the discussion above. The corresponding bigons bounded by the curves $\g_1'$ and $\g_2$ are denoted by $u_k'$ having corners $p_i'$ and $q_i'$ which are $C^1$-close to the corners of $u_k$. We have to show $$\mathop{\sum}\limits_{i,j=0}^{n-1}m_2(q_j',p_i)=\pm p$$
Indeed, the only term that survives on the left hand side is the one corresponding to the bigon $u_k$ on which $p$ lies. Thus the result is $p$ upto a sign.
\end{proof}

\begin{prop}\label{C^1}
Let $\g_1$ and $\g_1'$ be $C^1$-close isotopic curves  bounding $m$ bigons as above. Then $f_{1,1'}$ is a closed morphism and,
for any curve $\b$ jointly transverse to $\g_1$ and $\g_1'$, $$m_2(f_{1,1'}, \cdot):CF(\b,\g_1)\to CF(\b,\g_1')$$ is an isomorphism of complexes.

\end{prop}

\begin{proof}
  Since $\g_1$ and $\g_1'$ are $C^1$-close this gives a bijection between points in $\b\cap\g_1$ and $\b\cap\g_1'$, and consequently an isomorphism of chain complexes $CF(\b,\g_1)\cong CF(\b,\g_1')$. 
\end{proof}
\begin{prop}\label{jointly_outside}
For any $\b$ jointly transverse to $\g_1$ and $\g_1'$ outside the neighbourhood of a bigon between $\g_1$ and $\g_1'$
 where it has two additional intersection points $\{a,b\}$ with one of the curves (see Figure \ref{jointly_local}), the map $m_2(\cdot,f_{1,2})$ induces an isomorphism in homology.

\begin{figure}[ht]
\def\svgwidth{0.55\textwidth}
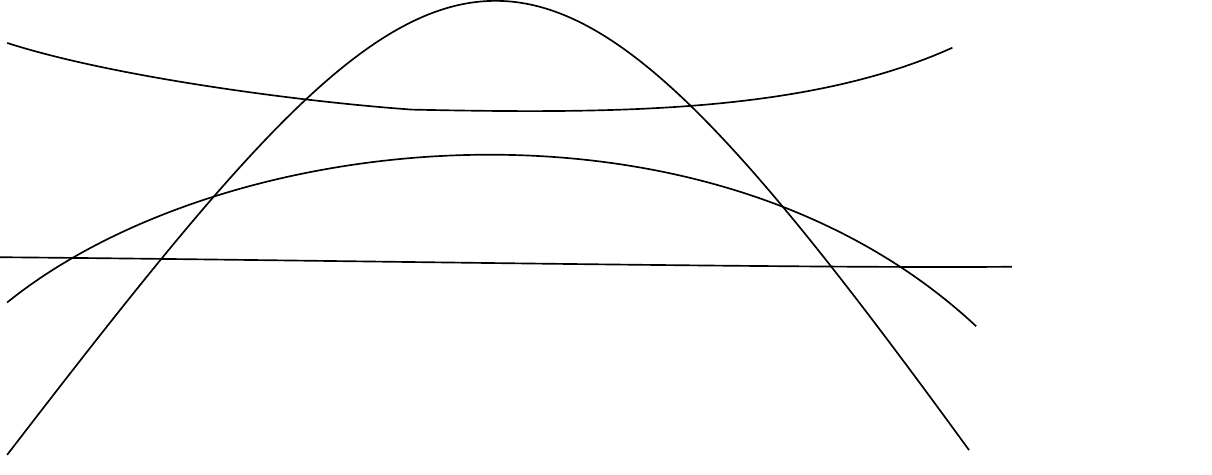
\caption{}
\label{jointly_local}
\end{figure}

\end{prop}
\begin{proof}
If $\b$ is jointly transverse to $\g_1$ and $\g_2$ in the complement of one bigon where it has two additional intersection points $\{a,b\}$ with one of the curves, say $\g_1.$  We can find another curve $\g_1'$ which is $C^1$-close to $\g_1$ such that $\b$ is jointly transverse to $\g_1,\g_1'$. There is thus an isomorphism of complexes $CF(\b,\g_1')\cong CF(\b,\g_1)$.  This means that the composition $m_2(m_2(\cdot, f_{1,2}),f_{1,1'})$ is an isomorphism of complexes which implies injectivity of $m_2(\cdot, f_{1,2})$. In the absence of higher polygons, the product $m_2$ is associative and injectivity at the level of homology follows. Surjectivity can also be shown by a similar argument as well as isomorphism for $m_2(f_{1,2},\cdot)$.
\end{proof}

We generalise the above statement for any isotopic closed curves and show that there is an isomorphism at the level of homology.
\begin{theorem}\label{isotopic}
If $\g_1$ is isotopic to $\g_2$ then these curves are quasi-isomorphic objects of the Fukaya category.
\end{theorem}
\begin{proof}
If $\g_1$ is isotopic to $\g_2$, then there exists a sequence of isotopic curves \\$\g_1=\g_1^1, \ldots, \g_1^k=\g_2$ such that $|\g_1^i\cap\g_1^{i+1}|=2$. This reduces us to the case  where \mbox{$|\g_1\cap\g_2|=2$.} 

Let $\b$ be a curve transverse to $\g_1 $ and $\g_2$. Then we may have the following three possibilities for $\b$. First possibility is that $\b$ is jointly transverse to  $\g_1 $ and $\g_2$ and we are done by Proposition \ref{jointly_prop}. Next we consider the situation when $\b$  is jointly transverse to  $\g_1 $ and $\g_2$ outside a bigon
 where it has two additional intersection points $\{a,b\}$ with one of the curves, by Proposition \ref{jointly_outside} we have the required quasi-isomorphism. 
Lastly we consider $\b$ which has neither of the above properties. In this case we can add a sequence of extra curves $\g_1=\g_1^1,\ldots,\g_1^k=\g_2$ such that for each $1\leq i\leq k$ we have $|\g_1^i\cap\g_1^{i+1}|=2$ and $\b$ is either jointly transverse to a pair of consecutive curves $\g_1^i $ and $\g_1^{i+1}$ or it will be jointly transverse to such a pair outside a bigon 
where it has two additional intersection points $\{a,b\}$ with one of the curves. This can be arranged because every isotopy between two curves on a surface can be realised by adding or removing bigons. In either situation, we have for each $1\leq i\leq k$ induced isomorphisms in homology, since each intersection point $f_{i,i+1}\in CF^0(\g_1^i,\g_1^{i+1})$ is a closed morphism for all $1\leq i\leq k$, thus $\g_1^i $ is quasi-isomorphic to $\g_1^{i+1}$.  In the absence of higher polygons, the quasi-isomorphism between $\g_1$ and $\g_2$ is simply obtained by composition of these intermediate quasi-isomorphisms between $\g_1^i $ and $\g_1^{i+1}$ for each $i$.

\end{proof}

\section{The Grothendieck Group $K_0(Fuk(\S))$}

In this section we construct a map $\Phi$ from the Grothendieck group of the Fukaya category $$K_0(Fuk(\S))\mathop{\to}\limits^{\Phi} H_1(S\S;\mathbb{Z})$$ where $S\S$ is the unit tangent bundle over $\S$ and the domain of $\Phi$ is the Grothendieck group of the derived Fukaya category of $\S$. This map  will take a curve in the Fukaya category to its lift in $S\S$ defined by the unit tangent vector.
Let $X$ be a non singular vector field on $\S \setminus z$.

 Using $X$, the tangent space of $\S\setminus z$ is trivialized and this gives a splitting
$$H_1(S\S;\mathbb{Z})\cong H_1(\S;\mathbb{Z})\oplus \mathbb{Z}/\chi(\S).$$
We start with a map $\phi$ defined on objects of the Fukaya category by $$Fuk(\S)\to H_1(\S,\mathbb{Z})\oplus \mathbb{Z}  / \chi(\S)$$ $$\g
\mathop{\to}\limits^{\phi}  ([\g], wd_X(\g))\ .$$ Note that the winding number of a curve $\gamma$ modulo $\chi(\S)$ does not change if $\g$ is
 isotoped over $z$, so that we may safely assume that $\gamma$ is generic with respect to $z$.
The map $\phi$ is easily seen to be surjective. The goal of this section is to show that this maps $\phi$ extends to the Grothendieck group which defines the expected map $\Phi$.

We use twisted complexes over $Fuk(\S)$ to derive this $A_{\i}$-precategory. Recall that the category of twisted complexes over a given $A_{\infty}$-precategory $\mc{A}$ is itself an $A_{\infty}$-precategory with twisted complexes over $\mc{A}$ as its objects and with morphisms between transversal pairs defined as follows.

We use the notation $\Nn$ for the set $\{0,1,2,\dots,n\}$. A twisted complex $\Gamma$ is a finite sequence of unobstructed immersed closed curves $(\g_k)_{k\in \Nn}$ 
 equipped with degree 1 morphisms $\Delta=(\Delta_{k,l})_{k<l}$ such that $\hat{m}_1(\Delta)=0.$  For the sake of convenience, we briefly recall the definition of the twisted differential $\hat{m}_1$ (for details see \cite{fukaya1993morse}, \cite{seidel2008fukaya}). Given twisted complexes $A=(A_i,\Delta^A)$ and $B=(B_j, \Delta^B) $ such that the sequence of objects $(A_{i_1},\ldots,A_{i_c},B_{j_1},\ldots, B_{j_d})$ is transverse, $\hat{m}_1$ is the differential induced by $\Delta^A$ and $\Delta^B$ on the graded algebra $\mc{M}or(A,B)= \mathop{\bigoplus}\limits_{i,j}\mc{M}or(A_i,B_j).$
\begin{defn}
Given an increasing sequence of integers $I$, let $|I|$ denote the number of elements in $I$,then
$$\hat{m}_1(F)=\mathop{\sum}\limits_{I,J} m_{|I|+|J|+1}(\Delta_J^B,f_{i,j},\Delta_I^A)$$ for all $F=(f_{i,k})\in \mc{M}or(A,B)$, where the sum is taken over all meaningful choices of sequences $I,J$.
\end{defn}

We extend  the map $\phi$ to twisted complexes in the obvious way
$$Tw(Fuk(\S))\to H_1(\S;\mathbb{Z})\oplus \mathbb{Z} / \chi(\S)$$
$$\G=(\g_k)_k\mapsto ([\G], wd_X(\G))=\mathop{\sum}\limits_k ([\g_k],wd_X(\g_k))$$ where $\G$ is an arbitrary twisted complex over $Fuk(\S)$.
\subsection{$C^1$-close twisted complexes}
Let $(\Gamma,\Delta)$ be a twisted complex where $$\G=(\g_k)_{k\in \Nn}\ ,
 \ \Delta=(\Delta_{k,l})_{k<l}\ .$$

 For a fixed index $k$, consider the finite set of intersection points between $\g_k$ and $\g_l$, for all $l$, equipped with the natural $\mathbb{Z}_2$-grading in $CF(\g_k,\g_l)$. These points subdivide the curves $\g_k$ into segments. Let us denote the set of these segments by $\mc{I}_k$. We choose  $\g_k'$, which is a  $C^1$-perturbation of $\g_k$ with the following properties:
\begin{enumerate}
  \item The curves $\g_k$ and $\g_k'$ intersect twice in each $I\in \mc{I}_k$. If we orient the segment $I$ according to the orientation of $\g_k$, then the degree 0 intersection point occurs before the intersection point of degree 1.
  \item The curve $\g_k'$ is parallel to $\g_k'$  outside a disc neigbourhood of each intersection point $c$, where it follows the local model in Figure \ref{local}.
  \end{enumerate}
\begin{figure}[ht]
\def\svgwidth{0.4\textwidth}
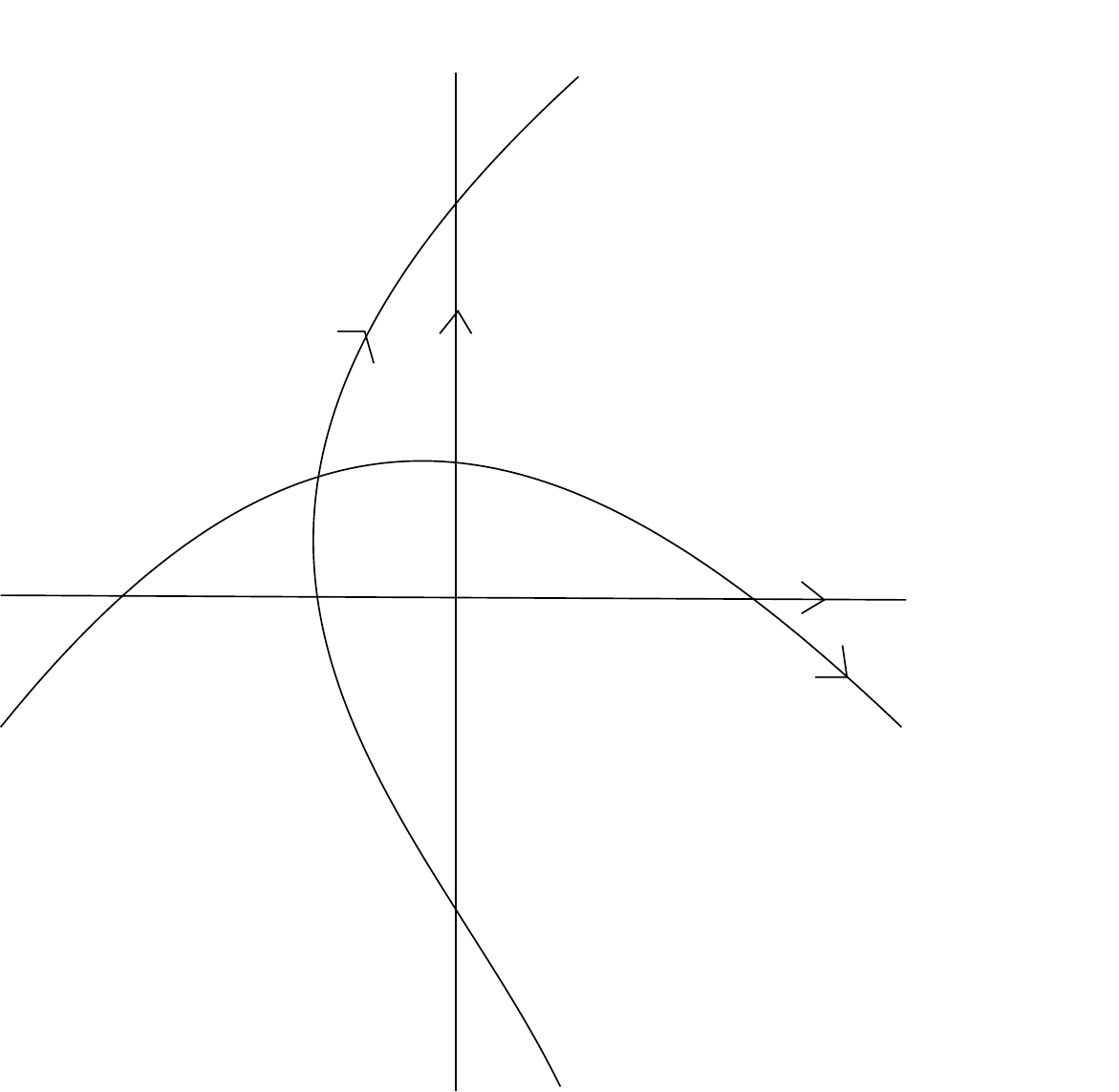
\caption{}
\label{local}
\end{figure}

In terms of the intersection points $c$ between $\g_k$ and $\g_l$, we decompose the morphisms $\Delta_{k,l}$ as $$\Delta_{k,l}=\mathop{\sum}\limits_{c\in\g_k\cap \g_l} \delta_{c}\, c.$$
By definition of $\Delta$, the coefficients in $\Delta_{k,l}$ vanish whenever the degree of $c$ in $CF^*(\g_k,\g_l)$ is different from 1.
 Each $c'$ in $\g_k'\cap \g_l'$ corresponds by $C^1-$closeness to a unique intersection point $c$ in $\g_k\cap \g_l$. Let $$\Delta_{k,l}'=\mathop{\sum}\limits_{c'\in\g_k'\cap \g_l'} \delta_{c}\,c'.$$
We denote by $\Gamma'$ the collection of curves $\g_k'$ and by $\Delta '$ the collection of morphisms $\Delta_{k,l}'.$

For each $k\in\Nn$, let $f_k$ be the quasi-isomorphism between $\g_k$ and $\g_k'$ similar to the explicit one constructed in Section 4: $$f_k=\mathop{\sum}\limits_{p\in \g_k'\cap \g_k,\deg{p}=0} p\in CF^0(\g_,\g'_k).$$ Taking a sum over all curves $\g_k$ we define: $$ F=\mathop{\oplus}\limits_{k}f_k\in \mathop{\bigoplus}\limits_{k}CF^0(\g_k,\g_k')\subset CF^0(\Gamma,\Gamma'). $$
\begin{lem}
a) $(\Gamma',\Delta')$ is a twisted complex.\\
b)  $F\in CF^0(\G,\G')$ is a quasi-isomorphism, in particular it is closed.\\
\end{lem}
 \begin{proof}
Statement a) is \cite[Lemma 6.2]{abouzaid2008fukaya}. Closeness of $F$ is \cite[Lemma 6.3]{abouzaid2008fukaya}. The quasi-isomorphism is mentionned in 
\cite[Remark 6.4]{abouzaid2008fukaya}. Here we do not care about areas and the situation is slighty simpler. For the convenience of the reader we reproduce briefly the proofs here.
 
 Figure \ref{local} is an accurate local model at $c\in\g_k\cap \g_l$, showing bigons, triangles and squares contributing to $\hat{m}_1(\Delta)$. Note that this means for $CF^*(\cdot,\cdot)$  and all input intersection points, the source is the horizontal curve and target is the vertical curve whereas for the output point it is the opposite.
It can be seen geometrically that every polygon $u$ contributing to $\hat{m}_1(\Delta)$ can be perturbed to a polygon $u'$ by going along curves $\g_k'$ instead of $\g_k$. We have $\hat{m}_1(\Delta')=\hat{m}_1(\Delta)$, since $(\Gamma,\Delta)$ is a twisted complex $\hat{m}_1(\Delta')$ also vanishes which proves a)

For statement b), we first show that $F$ is closed morphism. 
The curves $\g_l$ and $\g_l'$ being $C^1$-close with respect to $\g_k$, the product map $$m_2(f_l, \cdot): CF(\g_k,\g_l)\to CF(\g_k,\g_l')$$ is an isomorphism of vector spaces hence invertible. Also the maps $$m_{s+1}(x_s',\ldots,x_l',f_l,x_{l-1},\ldots ,x_1)=0,\,\,\,\,\,\,\,\,\, \mbox{ for any }s>2,$$ since all $\g_i,\g_i'$ are $C^1$-close to each other relative to $\g_j,\g_j'$ for $i\neq j.$ So the only polygons contributing to $\hat{m}_1(f_l)$ are the ones contributing to $m_1$ and $m_2$. In order to show $\hat{m}_1(f_l)=0$ we need to show
$$\begin{array}{l}
                                                                  m_1(f_l)=  0 \\
                                                                  m_2(f_l,\Delta_{k,l})+m_2(\Delta_{k,l}',f_k) =  0
                                                                \end{array}
$$The first holds since $f_l$ is a quasi-isomorphism of curves $\g_l,\g_l'$, hence is closed. For each $p_l$, there exists at most one intersection point $c\in \g_k\cap \g_l$  adjacent to $p_l$, that contributes to the first term in the second condition. Given a degree zero intersection point $p_k$ between $\g_k$ and $\g_k'$ which is adjacent to $c$, we have $$m_2(p_l,c)+m_2(c',p_k) =  0 .$$ The computation reduces to the local model at $c$ shown in Figure \ref{local} and it can easily be checked that the triangles corresponding to each of these terms contribute opposite signs. 

For proving statement b), by \cite[Lemma D.6]{abouzaid2008fukaya} it is enough to show that the following induces isomorphisms in homology for any unbstructed curve $\beta$ transverse to $\G$, $\G'$:
$$\hat m_2(\ ,F): CF(\G,\beta)\to CF(\G',\beta)\ ,$$
$$\hat m_2(F,\ ): CF(\beta,\G)\to CF(\beta,\G')\ .$$
Any curve  $\beta$ can be isotoped to a curve $\beta'$
which is jointly transverse to $\G$, $\G'$ outside neighbourhoods of the intersection points $h\in \gamma_k\cap \gamma_l$.
 As in Proposition \ref{jointly_prop}, we get
  isomorphisms of complexes
 $$\hat m_2(\ ,F): CF(\G,\beta')\simeq CF(\G',\beta')\ ,$$
$$\hat m_2(F,\ ): CF(\beta',\G)\simeq CF(\beta',\G')\ .$$ 
 Using a sequence of quasi-isomorphisms between $\beta'$ and $\beta$ 
 we obtain the expected quasi-isomorphisms. 
\end{proof}
\subsection {Zero objects in the Grothendieck group} Here we identify the zero objects in the derived Fukaya category.
\begin{prop}If $(\G,\Delta)$ is a twisted complex quasi-isomorphic to the zero object then, $$([\G],wd_X(\G))=0\in H_1(\S,\mathbb{Z})\times\Z/\chi(\S)\ .$$
\end{prop}
\begin{proof}

The vanishing of $[\G]\in H_1(\S;\mathbb{Z})$ can be obtained simply by using the fact that for any simple closed curve $\g$ we have $HF(\G,\g)=0$, and hence
the intersection $[\G]\cdot [\g]$, which is equal to the Euler characteristic of $CF(\G,\g)$ vanishes for every $\g$. This implies $[\G]=0$. However we will need an explicit $2$-chain realising this null homology in order
to study the winding number. This construction will follow Abouzaid similar proof \cite{abouzaid2008fukaya} with simplification coming from the fact that we do not need areas. For our purpose we  use slightly different notation.

Let $(\Gamma,\Delta)$ be a twisted complex where $$\G=(\g_k)_{k\in \Nn}\ ,
 \ \Delta=(\Delta_{k,l})_{k<l}\ .$$
 We define $(\Gamma',\Delta')$ and $F\in CF(\G,\G')$ as in the preceding subsection. By definition of a twisted complex, we have
  $\hat{m}_1(\Delta)=0$.
  
For $K\subset \Nn$, $K=\{k_0<k_1<\dots<k_\nu\}$, we denote by $\mathcal{C}(K)$ the set of sequences of intersection points $(c_\nu,\dots,c_1)$ where $c_j\in \gamma_{k_{j-1}}\cap\g_{ k_j}$.
We will denote by $\nu(K)$ the length of a sequence $C\in \mathcal{C}(K)$,
that is, $1+\nu(K)$ is the cardinality of $K$.
 In this notation we have
$$\hat{m}_1(\Delta)=\mathop{\sum}\limits_{K\subset\Nn,C\in \mathcal{C}(K)}m_{\nu(K)}(C)\ .$$
So the polygons which contribute to  $\hat{m}_1(\Delta)$ are those in
$$\mathcal{M}(x,C)\ \text{for}\ C\in \mathcal{C}(K)\ ,\ K=\{k_0<k_1<\dots<k_\nu\}   \subset \Nn\ ,
\ x\in \g_{k_0}\cap \g_{k_\nu}\ .$$

The assumption $\Gamma$ is quasi-isomorphic to the zero object implies that the closed morphism $F\in CF^0(\G,\G')$ is a boundary, which means that there exists \mbox{$H\in CF^1(\Gamma,\Gamma')$} such that $\hat{m}_1(H)=F$. 

We denote by $\mathcal{C}'(K)$ the set of sequences of intersection points $(c_\nu,\dots,c_1)$ where $c_j'\in \gamma'_{k_{j-1}}\cap\g'_{k_j}$.
The set of polygons contributing to $\hat{m}_1(H)$ are those in 
$$\mc{M}(x,C',h,C)
\text{ where }
\left\{\begin{array}{l}
h\in \gamma_k\cap \gamma'_l\\
C\in \mc{C}(K)\ , \ K=\{k_0<k_1<\dots<k_{\nu}\}\\
C'\in \mc{C}(K')\ , \ K'=\{l_0<l_1<\dots<l_{\mu}\}\\
x\in \gamma_{k_0}\cap \gamma'_{l_{\mu}}\ 
\end{array}\right.
$$
In the case $x=p_j$, we have $k_0=l_{\mu}=j$ and we can write
$$L=K'\cup K=\{l_0<\dots<l_{\mu}=j=k_0<l_{\mu+1}=k_1<\dots<l_{\mu+\nu}=k_{\nu}\}\ .$$
In particular we have $l_0<k_{\nu}$.
Thus the set of polygons contributing to $\langle\hat{m}_1(H),p_j\rangle$ are those in 
$$\mc{M}(p_j,C',h,C)
\text{ where }
\left\{\begin{array}{l}
h\in \gamma_k\cap\gamma_l'\ ,\ k<l\\
C\in \mc{C}(K)\ , \ K=\{j=k_0<k_1<\dots<k_{\nu}=k\}\\
C'\in \mc{C}(K')\ , \ K'=\{l=l_0<l_1<\dots<l_{\mu}=j\}\\
\end{array}\right.
$$
There is a canonical bijection between the sets $\g_k\cap\g_l' $ and $\g_k\cap\g_l$. Let $\overline{h}\in \g_k\cap\g_l$ be the intersection point corresponding to the degree $1$ intersection point  $h\in \g_k\cap\g_l'\subset CF(\g_k,\g_l')$. Then  $\overline{h}$ has degree $0$ in $CF^*(\g_l,\g_k)$ (see Figure \ref{fig:key}).

\begin{figure}[ht]
\def\svgwidth{0.4\textwidth}
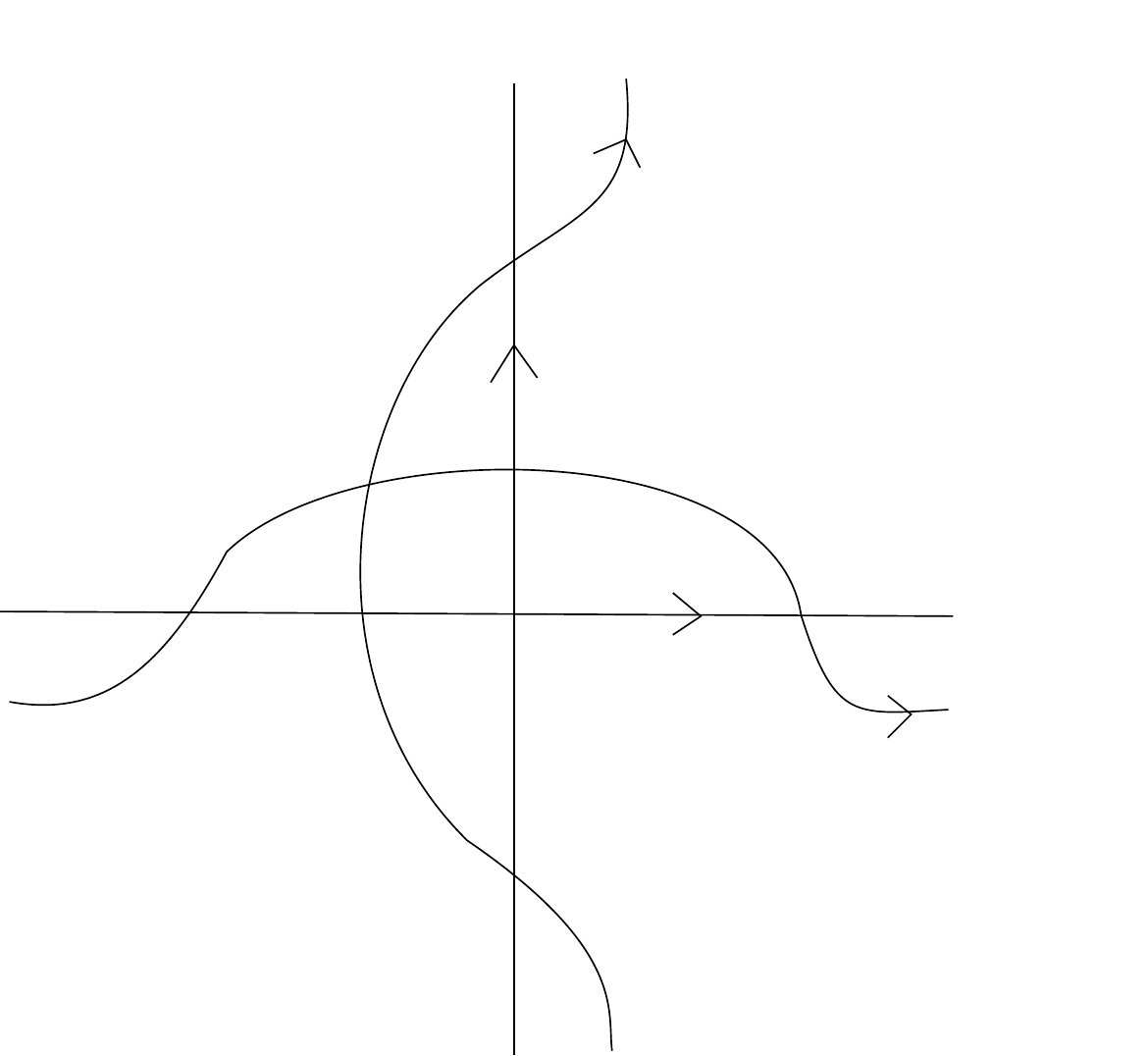
\caption{}
\label{fig:key}
\end{figure}

A key fact here is that  each polygon in the above family corresponds to a unique one in those described previously for contribution to $\hat{m_1}(\Delta)$. More formally, for each $h$, $j$, $K$, $K'$ as above, there exists a  map $\Psi$
\begin{equation*} \mc{M}(p_j,C',h,C)\mathop{\to}\limits^{ \Psi }\mc{M}(\overline h,D=C'\cup C)\end{equation*}
which is defined as follows.

Let $v$ be a polygon  in $\mc{M}(p_j,C',h,C)$. Ordered corners of $v$ are given by the sequence $(p_j,c'_{\mu},\dots ,c'_{1},h,c_{\nu},\dots,c_1)$, then $u=\Psi(v)$ is obtained
by replacing the two consecutive corners $p_j$, $c'_{\mu}$, by the neighbouring corner $c_{\mu}$ and each $c'_j$ to the near point $c_{j} $. Between each of the two shifted points, the boundary of the disc moves from the arc in $\G'$ to the corresponding neighbouring arc in $\Gamma$.
The result is a disc with  boundary on $\Gamma$ and ordered vertices 
$(\overline h, d_{\mu+\nu},\dots,d_1)$, where $d_{\mu+i}=c_i$ for $1\leq i\leq \nu$, and
for $1\leq s\leq \mu$, $d_s\in \g_{l_{s-1}}\cap \g_{l_s}$  is closed to $c'_s$.

Conversely, for $u\in \mc{M}(\overline h,D)$, the choice of a degree $0$ point $p_j\in \g_j\cap \g_j'\subset
CF(\g_j,\g'_j)$ gives a splitting $D=C\cup C'$,  and an element $z$ in the  inverse image  $u^{-1}(\overline p_j)$ determines a unique polygon $v\in \mc{M}(p_j,C',h,C)$. We get a bijection
$\widetilde{\Psi}$
\begin{equation*} \mc{M}(p_j,C',h,C)\mathop{\to}\limits^{ \widetilde\Psi }\mc{\widetilde M}(\overline h,D)=\{(u,z)\ ,u\in \mc{ M}(\overline h,D)\ ,\ z\in u^{-1}( p_j)\}
\end{equation*}

We want to construct from the polygons contributing to $\hat m_1(\Delta)$ a 2-chain $U$ whose boundary is $\mathop{\sum}\limits_{k}\g_k$. Each polygon $u$ will be weighted with an appropriate multiplicity.
 Let $\mathcal{U}$ be the union of all $\mathcal{M}(\overline h,D))$ for all $h$ and $D$ as above. Note that $u$ determines $h$ and $D$. Form the previous analysis of polygons, the degree $1$ element $H$ whose boundary is the canonical quasi-isomorphism $F$ can be written as,
 $$H=\mathop{\mathop{\sum}\limits_{k<l}}\limits_{h\in \g_k\cap\g'_l\,,\, \deg(h)=1}\ \lambda_h\, h\ .$$
 For $u\in \mc{U}$, let $M(u)=\lambda_{h}\,\delta_{D}\,(-1)^{s(u)} $.
 Here $\delta_D$ is the product of the coefficients $\delta_d$ for all points $d$ in $D$.
  The $2$-chain $U$ is defined by
  $$U=\mathop{\sum}\limits_{u\in\mc{U}}M(u)\,\, u.$$
 
  We will achieve the proof with the two following lemmas.
  \end{proof}

  For each degree $0$ intersection point $g\in CF^0(\g_j,\g_j')$
 $$\lan \hat{m}_1(\Delta),\overline g\ran=\mathop{\sum}\limits_{K\subset\Nn,D\in \mathcal{C}(K)} \ \ \mathop{\sum}\limits_{u\in \mathcal{M}(\overline g,D)} \delta_D\, (-1)^{s(u)} \ .$$
 
\begin{lem}\label{boundary_2_chain}
The boundary of the 2-chain $U$ is represented by $\mathop{\sum}\limits_k \g_k.$
\end{lem}

\begin{proof}
If $u$ is a polygon with a boundary edge on $\g$, and $p$ is a point on $\g$, then we have a local degree  for the boundary of $u$  counting with sign
the points in the inverse image $u^{-1}(p)$, which we denote by $n_u(p)$.

Let $p\in \g_j$, then the required statement is equivalent to showing that $$\mathop{\sum}\limits_{u\in \mc{U}}M(u)n_u(p)=1.$$ The value of $n_u(p)$ is locally constant on each segment delimited by intersection points with other curves in $\G$, and this segment contains
a degree zero intersection point $p_j\in CF(\g_j,\g'_j)$.
It is enough to show that the coefficient of $p_j$ in the expression for $\hat{m}_1(H)$ equals $\mathop{\sum}\limits_{u\in \mc{U}}M(u)n_u(p_j)$. 
We have $n_u(p_j)=\sum_{z\in u^{-1}(p_j)}\deg_z(u)$ where the local degree is $\pm 1$ depending if $\partial u$ is locally oriented or not.
The bijection $\widetilde{\Psi}$ associates to $(u,z)$ a unique polygon $v\in\mc(p_j,C',h,C)$.
The polygon $v$  contributes a sign $s(v)$ to $\hat{m}_1(H)$. From the definition of signs in higher products, we have $s(v)=s(u)\deg_z(u)$. We can now compute
\begin{eqnarray*}
\sum_uM(u)n_u(p_j)&=&\sum_{u,z\in p^{-1}(p_j)}M(u)\deg_z(u)\\
&=&\sum _{u,z\in p^{-1}(p_j)}\lambda_h\delta_D (-1)^{s(u)}\deg_z(u)\\
&=&\sum_v\delta_{C'} \lambda_h\delta_C (-1)^{s(v)}\\
&=&\langle \hat m_1(H),p_j\rangle=1\\
\end{eqnarray*}
This completes the proof.
\end{proof}

In the following proof we will compute the winding number \cite{chillingworth1972windingI}. Note that our winding number $wd_X$ counts the rotation of the tangent vector relative to $X$ rather than the opposite.
It will be convenient to use the Euler measure. Here we follow \cite[Section 4]{lipschitz}.
The Euler measure of a surface $S$ with corners on the boundary is $\frac{1}{2\pi}$ times the
integral over $S$ of the curvature of a metric on $S$ for which the boundary  is geodesic and the corners  are right angles. We use the notation $e(S)$. An important property
is additivity for gluing along segments of the boundary. In the case of smooth boundary it coincides with the Euler characteristic. For a $n$-gon $D$ with convex corners, the value is $e(D)=1-\frac{n}{4}$.

\begin{lem}\label{q_isom_zero}
If $\Gamma$ is quasi-isomorphic to the zero object, then the winding number of $\Gamma$ vanishes, that is, $$wd_X(\Gamma)=0.$$
\end{lem}
\begin{proof}
We compute the Euler measure for the $2$-chain $U$:
\begin{eqnarray*}
e(U)&=&\sum_{u\in \mc{U}}M(u)e(u)\\
&=&\sum_{u\in \mc{U}}\ \lambda_h\delta_D(-1)^{s(u)}(1-\frac{1}{4}|D|)
\end{eqnarray*}
The vanishing of $\hat m_1(\Delta)$ implies
$$\sum_{u\in \mc{U}}\lambda_h\delta_D(-1)^{s(u)}=0\ .$$
We now have
$$e(U)=\frac{-1}{4}\sum_{k<l,g\in \g_k\cap\g_l}\ 
\sum_{u | g\in D} \lambda_h\delta_D(-1)^{s(u)}\ .$$
After describing the  polygons contributing to $\langle \hat m_1(H),g\rangle$, we get,
with computation similar to the previous lemma,
$$\sum_{u | g\in D} \lambda_h\delta_D(-1)^{s(u)}=\langle \hat m_1(H),g\rangle=0\ $$
The vanishing comes from the equality $\hat m_1(H)=F$ since $F\in \oplus_j CF(\gamma_j,\gamma_j')$.
We have obtained that the Euler number of $U$ vanishes. Since the boundary of $U$ has no corners
this implies that the winding number vanishes up to a multiple of the Euler characteristic corresponding to the polygons going over the singular point of the vector field $X$.
\end{proof}

\subsection{Cones and Connect sums}\label{cones}
In this section, we discuss the connection between twisted complexes and connect sums of curves on $\Sigma$, to be used later. We begin with the following definition.
\begin{defn}
If $\a$ and $\b$ are curves which intersect at a point $c$, then $\a\#_c\b$ is the immersed oriented curve obtained by resolving $c$.
\end{defn}
Note that this curve is only well-defined up to isotopy or equivalently, up to quasi-isomorphism in the Fukaya category.
\begin{figure}[ht]
\def\svgwidth{0.45\textwidth}
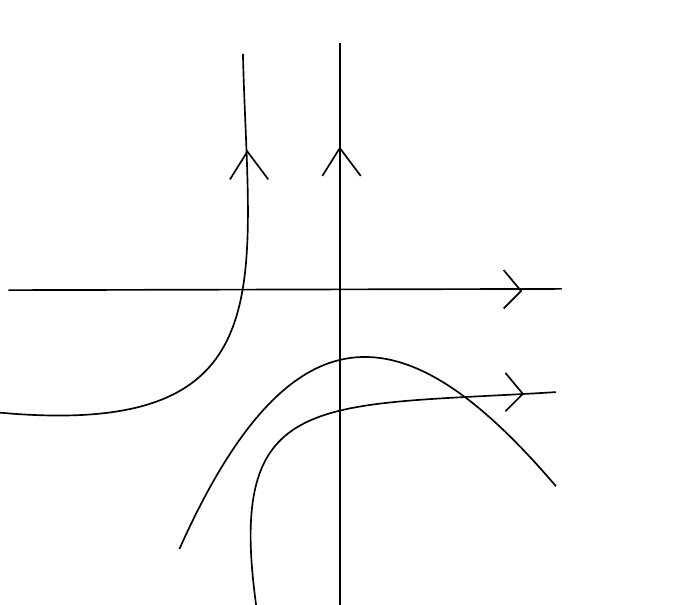

\caption{Connect sum of curves $\a$ and $\b$}
\label{connectsum}

\end{figure}

Since the category $Tw(Fuk(\S))$ is also an $A_{\infty}$ pre-category, for use in the subsequent paragraphs we state the formula for the product map $\hat{m}_2$: if $(A,\Delta^A),(B,\Delta^B)$ and $(C,\Delta^C)$ is a transverse triple of twisted complexes then $$\hat{m}_2(G,F)=\mathop{\sum}\limits_{K,J,I}m_{|K|+|J|+|I|+2}(\Delta_K^C,g_{l,k},\Delta_J^B,f_{i,j},\Delta_I^A)$$
for morphisms $F=(f_{i,j})\in \mc{M}or(A,B)$ and $G=(g_{l,k})\mc{M}or(B,C)$.
\begin{lem}\label{cone}
Let $\a$ and $\b$ be oriented immersed unobstructed curves that intersect transversally and minimally, and let the point $c$ be a degree 1 intersection between $\a$ and $\b$. Then $\a\#_c\b$ is quasi-isomorphic in $Tw(Fuk(\S))$ to the twisted complex $$Cone(c): \a\mathop{\longrightarrow}\limits^{c}\b.$$
\end{lem}
\begin{proof}
Let $a$ and $b$ be intersections of $\a\#_c\b$ with curves $\a\&\b$ respectively. As in Figure \ref{connectsum}, $a\in CF^0(\a\#_c\b,\a)$ this can be considered as $a\in CF^0(\a\#_c\b,Cone(c))$. Showing that $a$ is a closed element with respect to $\hat{m}_1$, is equivalent to showing     $$\begin{array}{c}
                                                                                                    m_1(a)=0 \\
                                                                                                    m_2(c,a)=0.
                                                                                                  \end{array}
$$
The first equality follows from the fact that $\a\&\b$ intersect minimally, so there can be no bigons with boundary on $\a\#_c\b$ and $\a$.
To show $m_2(c,a)=0$, we observe that there is a signed contribution of two triangles $u$ and $v$ which adds up to zero.
 The marked points have a contribution of 2 to $s(v)$ and 1 to $s(u)$, or a contribution of 3 to $s(v)$ and none to $s(u)$ (see Figure \ref{fig:markedpoints}). The remaining two cases are similar with $u\& v $ swapped. Either way one has $s(u)\equiv s(v)+1 \,(mod\,2)$ and the contributions cancel. Similarly, one may prove that $b$ is a closed element in $CF^0(Cone(c),\a\#_c\b)$. To show that $Cone(c)$ is quasi-isomorphic to $\a\#_c\b$, by Proposition 3.1, we need to show that $m_2(a,\cdot)$ is an isomorphism. Its injectivity follows from the following composition being an isomorphism:
$$\hat{m}_2(b,\hat{m}_2(a,\cdot)): HF^*(\g,\a\#_c\b)\to HF^*(\g,\a\#_c\b),$$ which is a consequence of the next lemma.
Based on similar arguments one can prove surjectivity of $m_2(a,\cdot)$ and even the isomorphism $m_2(\cdot,a)$.
\end{proof}

\begin{figure}
\begin{picture}(0,110)
 \put(-135,0){\def\svgwidth{0.35\textwidth}
 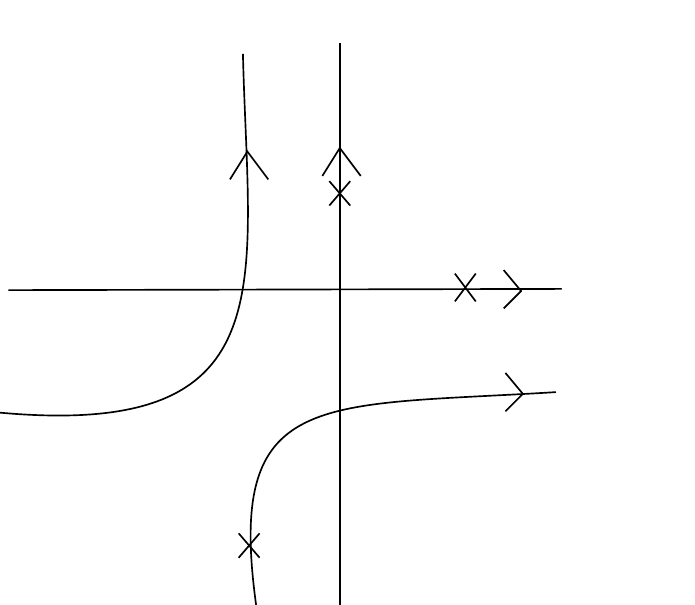} 

 \put(10,0){\def\svgwidth{0.35\textwidth}
 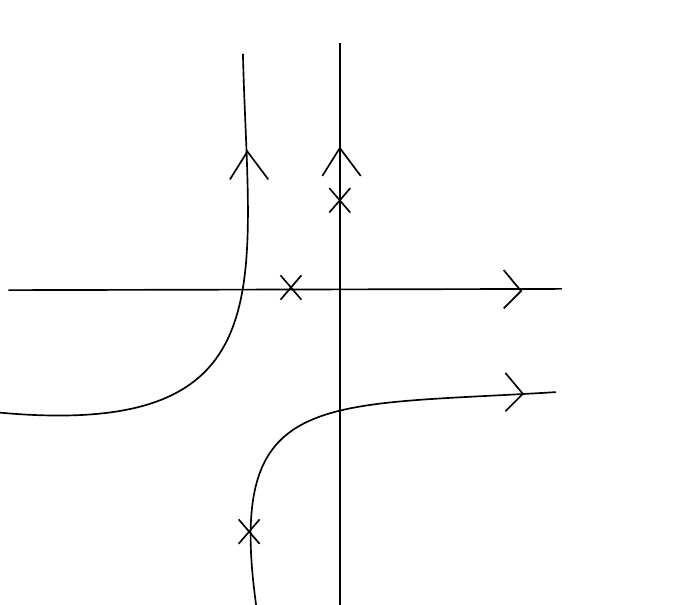} 
\label{markedpoints}
\end{picture}

\caption{Contributions of marked points to sign}
\label{fig:markedpoints}
\end{figure}

\begin{lem}
For a curve $\a\#_c\b$, which is $C^1$-close to $\a\cup \b$ away from a neighborhood of $c$, there is the following isomorphism of $\mathbb{Z}$-modules
$$\hat{m}_2(b,\hat{m}_2(a,\cdot)): CF^*(\g,\a\#_c\b)\to CF^*(\g,\a\#_c\b).$$\end{lem}
\begin{proof}
To every intersection point $x\in \g\cap \a\#_c\b$, by $C^1$-closeness, corresponds another intersection point $x'$ between $\g$ and $\a$ or $\b$ (see Figure \ref{connectsum}). Assume $x'$ lies on $\b$, then $\g$ divides the triangle $u$, into a square $u_1$ and a triangle $u_2$. We have
$$\begin{array}{rcl}
    m_3(c,a,x) &=& \pm x'  \,\,\,\,\&\\
    m_2(b,x') & = & \pm x.
  \end{array}
$$
In the simplest case these are the only contributions to the composition and thus establish the desired isomorphism.
\end{proof}
Iterative application of Lemma \ref{cone} gives the following:
\begin{prop}\label{prop cones twisted} Let $\a$ and $\b$ be transversally and minimally intersecting curves.
If $\{c_i\}_{i=1}^{n}$ and $\{b_j\}_{j=1}^{m}$ are the natural bases of $CF^ 1(\a,\b)$ and $CF^1(\a,\b[1])$ respectively, then the Dehn twist $\tau_{\b}^{-1}(\a)$ is quasi-isomorphic to the twisted complex
$$\a\mathop{\xrightarrow{\hspace*{3.2cm}}}\limits^{(c_1,\ldots,c_n,b_1,\ldots,b_m)} \b^{\oplus n}\oplus \b[1]^{\oplus m}.$$
\end{prop}
\noindent Note: Here $\b[1]$ denotes the curve obtained by changing the orientation of $\b$.
\subsection{Relations from mapping class group}
Next we seek to understand the kernel of the map $K_0(Fuk(\S))\to H_1(\S)$. We show that this is the subgroup in K-theory generated by collections of curves which bound a subsurface.
\begin{lem}\label{T is well-defined}
If $\g_1$, $\g_2$ are simple closed curves which bound homeomorphic submanifolds of $\S$, then they represent the same class in $K_0(Fuk(\S))$.
\end{lem}
\begin{proof}
Let $f$ be the homeomorphism between subsurfaces mapping $\g_1$ to $\g_2$. Since any homeomorphism can be written as a product of Dehn twists it suffices to show for a closed curve $\a$ in $\S$, the Dehn twist of $\g_1$ with respect to $\a$ has the same class in $K$-theory as $\g_1$. That is, one must show that $\tau_{\a}^{-1}(\g_1)$ and $\g_1$ are quasi-isomorphic. This follows using Proposition \ref{prop cones twisted}. Indeed, if $\g_1$ is separating, the intersection number of any curve $\a$ with $\g_1$ vanishes, hence the twisted complex (which is quasi-isomorphic to $\tau_{\a}^{-1}(\g_1)$) constitutes of just one term $[\g_1]$, therefore $\tau_{\a}^{-1}(\g_1)$ and $\g_1$ represent the same class in K-theory.
\end{proof}
The previous lemma generalizes to the case where $\G_1$ and $\G_2$ are a union of curves bounding homeomorphic submanifolds of $\S$, then $\G_1$ and $\G_2$ represent the same class in $K_0(Fuk(\S))$.
Let $T$ denote the class of a curve $\g$ bounding a subsurface of genus one. This class is well-defined: consider the complex $CF^*(\g,\g')$, where the curve $\g'$ is a push-off of $\g$ in $\Sigma$ and has two intersection points $p$ and $q$, so that $$CF^*(\g,\g')\cong \mathbb{Z}\lan p,q \ran.$$ There are two bigons, one with a positive and other with a negative coefficient. Thus $m_1=0$, and $$HF^*(\g,\g')\cong \mathbb{Z}p\oplus \mathbb{Z}q\cong \mathbb{Z}^2.$$ Thus $T$ is not the same as the class of zero in $K_0(Fuk(\Sigma))$. Also by the above Lemma if we transform this $\g$ (via homeomorphism) to another curve $\g''$ bounding another genus one subsurface, the class of these curves remains unchanged.

\begin{figure}
\begin{picture}(0,120)
 \put(-135,30){\def\svgwidth{0.35\textwidth}
 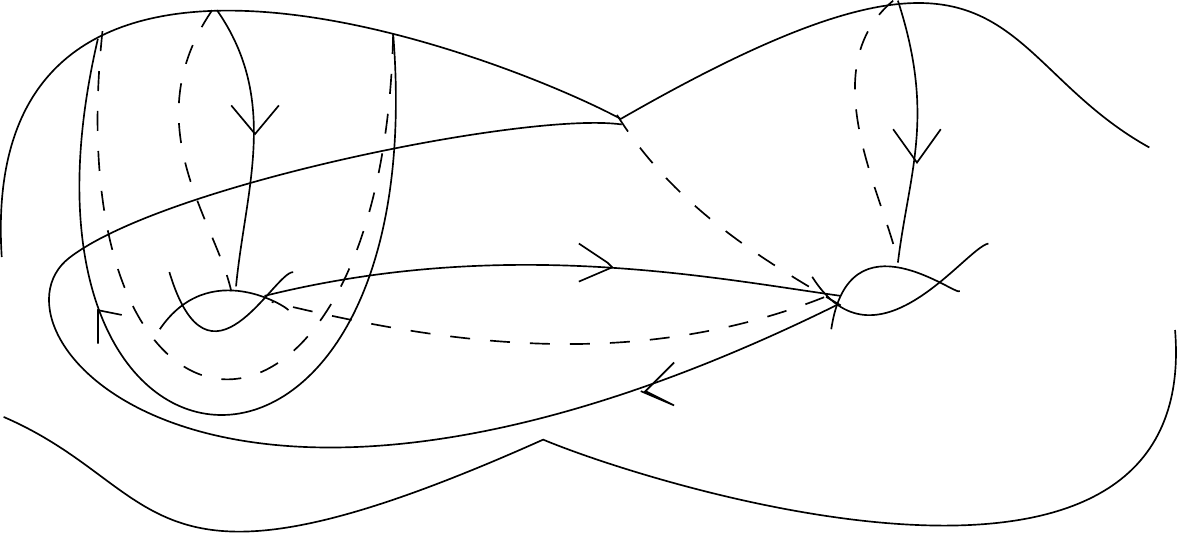} 
 \label{fig:Case_1}
 \put(10,30){\def\svgwidth{0.35\textwidth}
 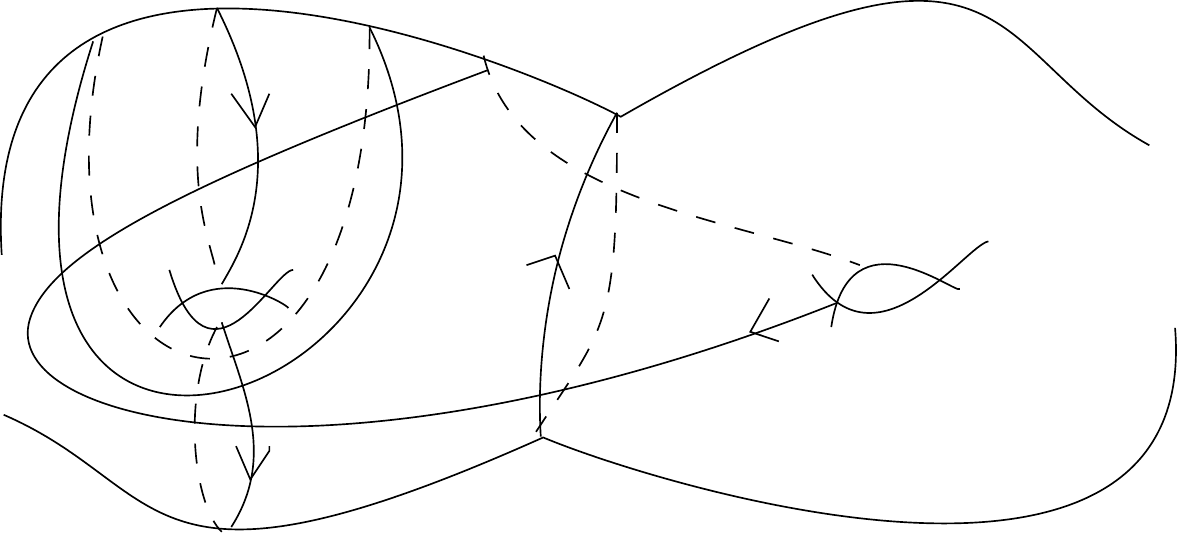} 
 \label{fig:Case_2}
 \put(-90,10){Case 1}
 \put(65,10){Case 2}

\end{picture}

\caption{Case 1: Non-separating curves, Case 2: One separating curve}
\label{fig:Case_1_2}
\end{figure}

\begin{lem}\label{pop lemma}
If $\a_1,\a_2 \& \a_3$ bound a pair of pants, then in K-theory
$$[\a_1]+[\a_2]+[\a_3]=T.$$
\end{lem}

\begin{proof}
Case 1: Suppose all three curves are non-separating curves. Let $\a_4$ be a curve bounding a torus with essential intersection points with $\a_2$. One can find a curve $\g$ as seen in figure. The idea is to successively add $\a_2$ to $\a_3$ and these to the curve $\g$ on one hand, and $\a_1[1]$ to $\a_4$ and then $\g$ on the other hand, to obtain isotopic curves. Given here, is a pictorial proof (for higher genus surfaces, see Figure \ref{fig:proof1}) of the isotopy of the following curves:
$$[\a_2]+[\a_3]+[\g]=-[\a_1]+[\a_4]+[\g]$$
$$\Rightarrow [\a_1]+[\a_2]+[\a_3]=[\a_4]=T,$$
 whereas just the curves are indicated for a surface of genus 2 (see Figure \ref{fig:genus2}).
\begin{figure}
\begin{picture}(0,260)

 \put(-135,170){\def\svgwidth{0.35\textwidth}
 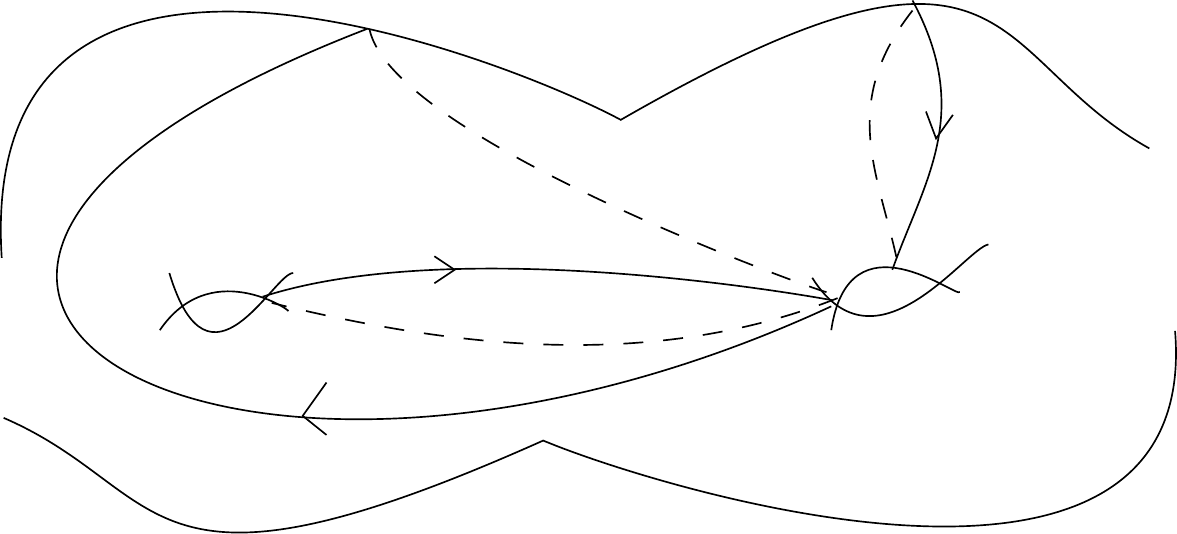} 

 \put(10,170){\def\svgwidth{0.35\textwidth}
 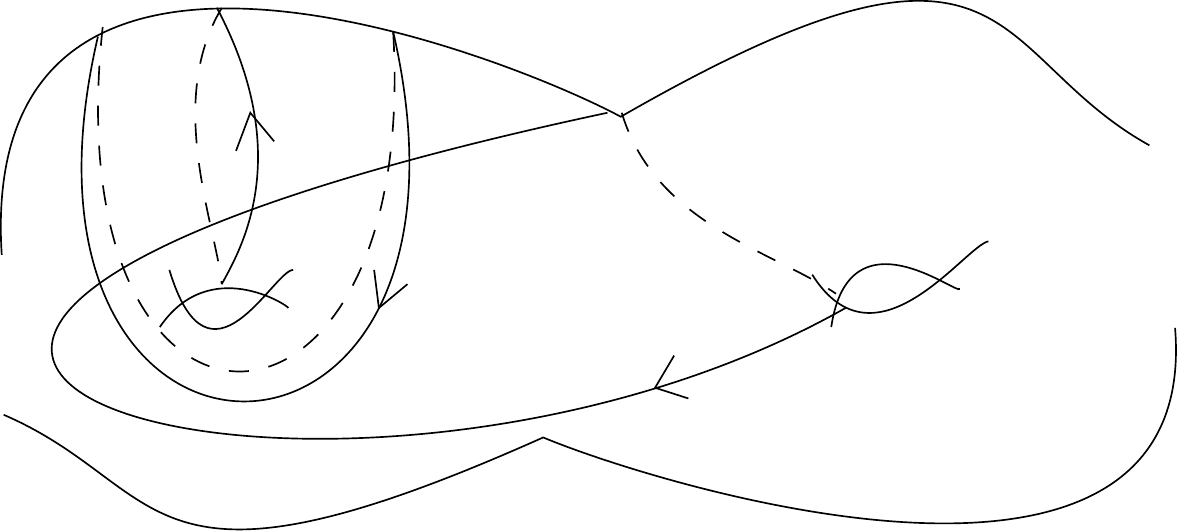} 

 \put(-135,90){\def\svgwidth{0.35\textwidth}
 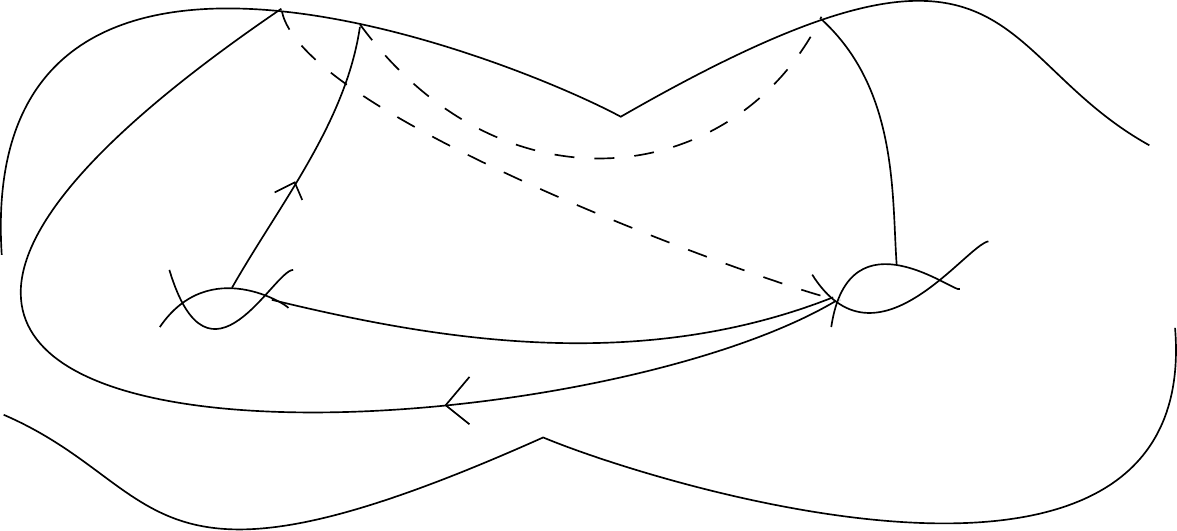} 

 \put(10,90){\def\svgwidth{0.35\textwidth}
 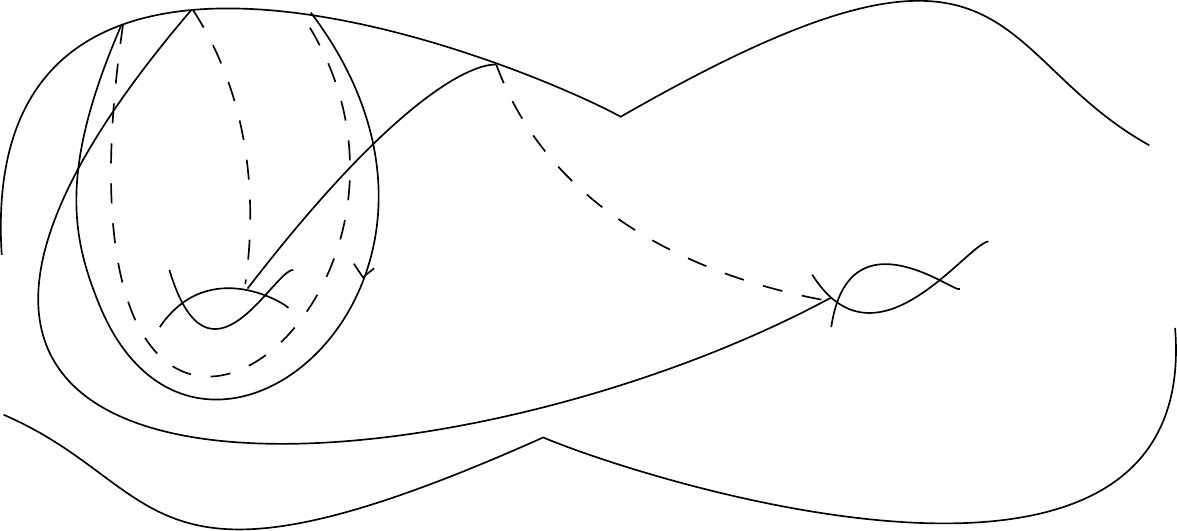} 

 \put(-135,10){\def\svgwidth{0.35\textwidth}
 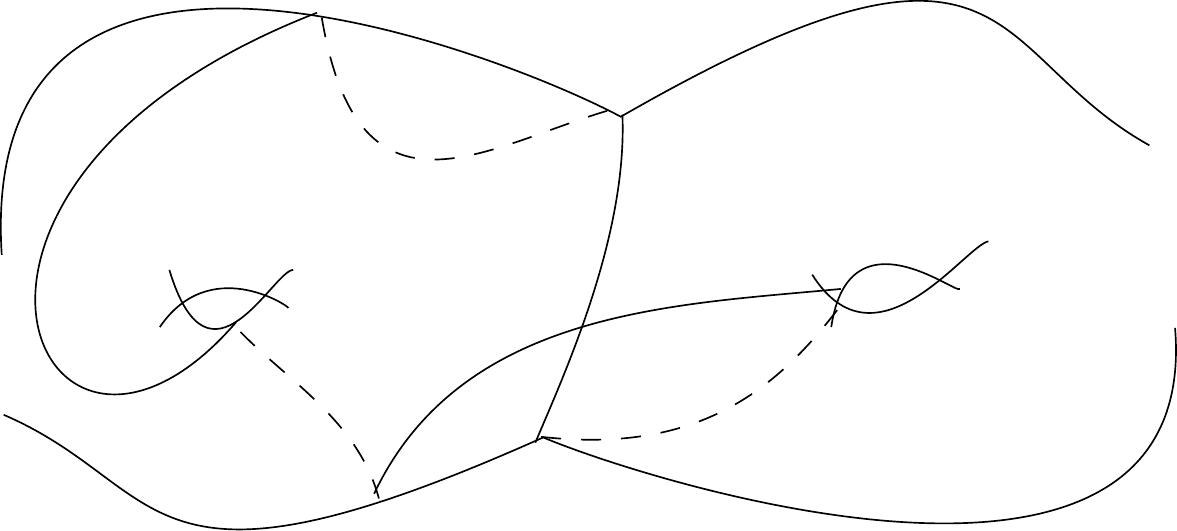} 

 \put(10,0){\def\svgwidth{0.5\textwidth}
 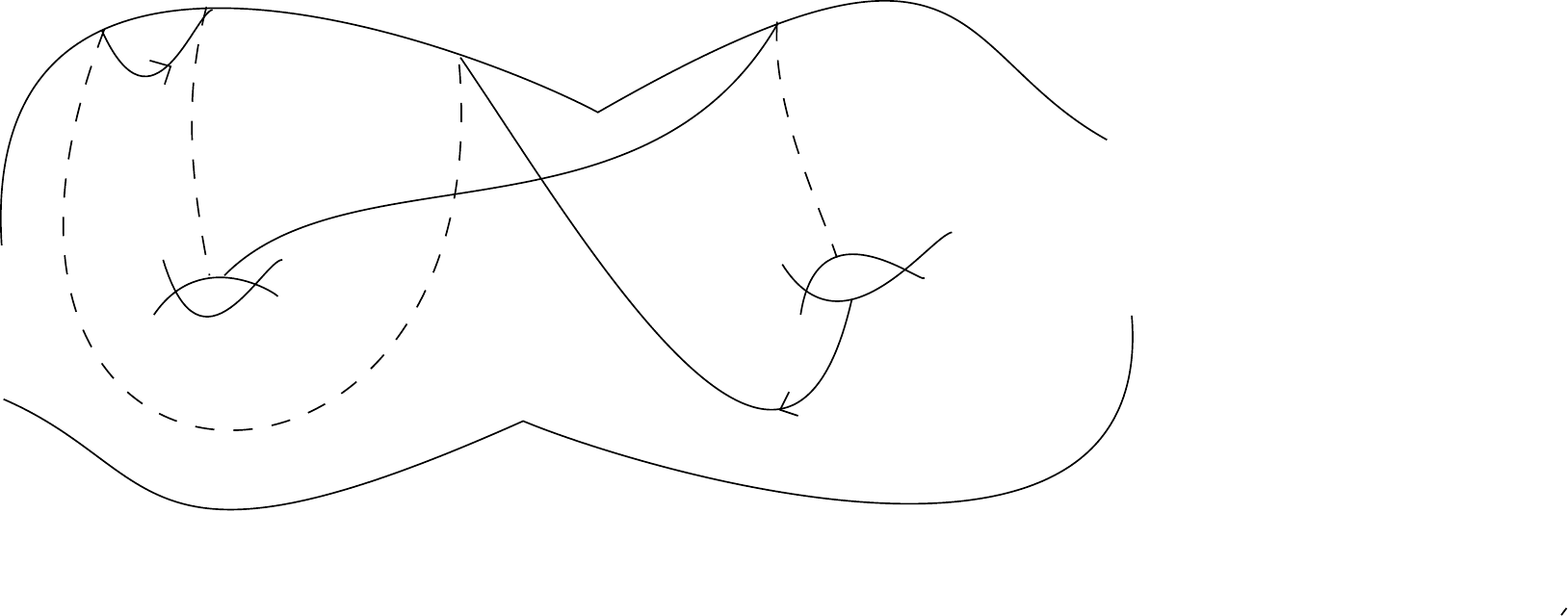} 

 \put(-100,-5){$[\a_2]+[\a_3]+[\g]$}
 \put(30,-5){$-[\a_1]+[\g]+[\a_4]$}

\end{picture}

\caption{Pictorial proof in the case of non-separating curves}
\label{fig:proof1}
\end{figure}

\noindent Case 2: One of the curves, say $\a_3$ bounds a subsurface.
Firstly we consider the situation where $\a_3$ bounds a subsurface of genus 1. This implies that $\a_1$ and $\a_2$ bound a cylinder and so their sum is zero by admissibility (see Remark \ref{admissibility}) and thus $[\a_3]=T$. In case $\a_3$ bounds a surface of higher genus, one can find curves $\a_4$ and $\g$ and show using similar arguments as in Case 1, the desired equality holds. A pictorial proof follows for the case of higher genus (see Figure \ref{fig:proof2}).
\begin{figure}
\begin{picture}(0,80)

 \put(-85,0){\def\svgwidth{0.5\textwidth}
 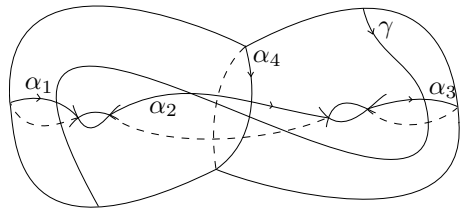} 

\label{genus2}
\end{picture}

\caption{Non-separating curves on a  surface of genus two.}
\label{fig:genus2}
\end{figure}
\begin{figure}
\begin{picture}(0,260)

 \put(-135,170){\def\svgwidth{0.35\textwidth}
 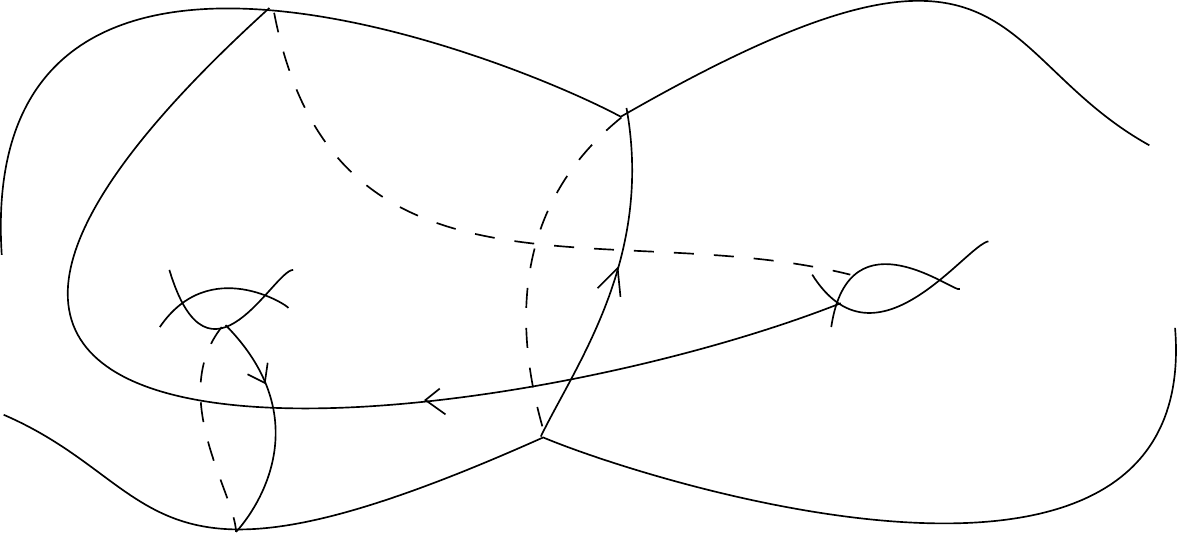} 

 \put(10,170){\def\svgwidth{0.35\textwidth}
 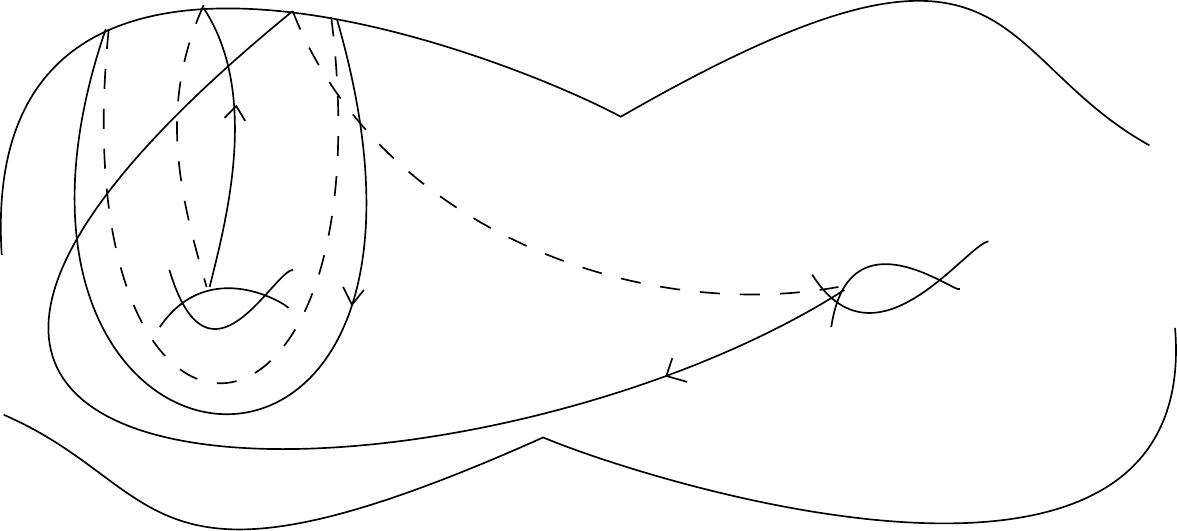} 

 \put(-135,90){\def\svgwidth{0.35\textwidth}
 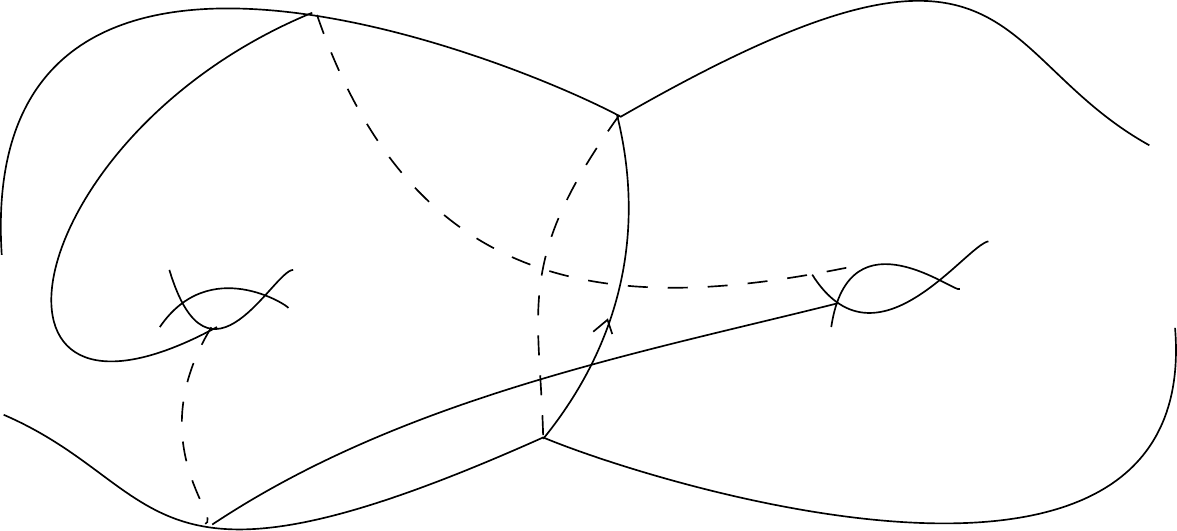}

 \put(10,90){\def\svgwidth{0.35\textwidth}
 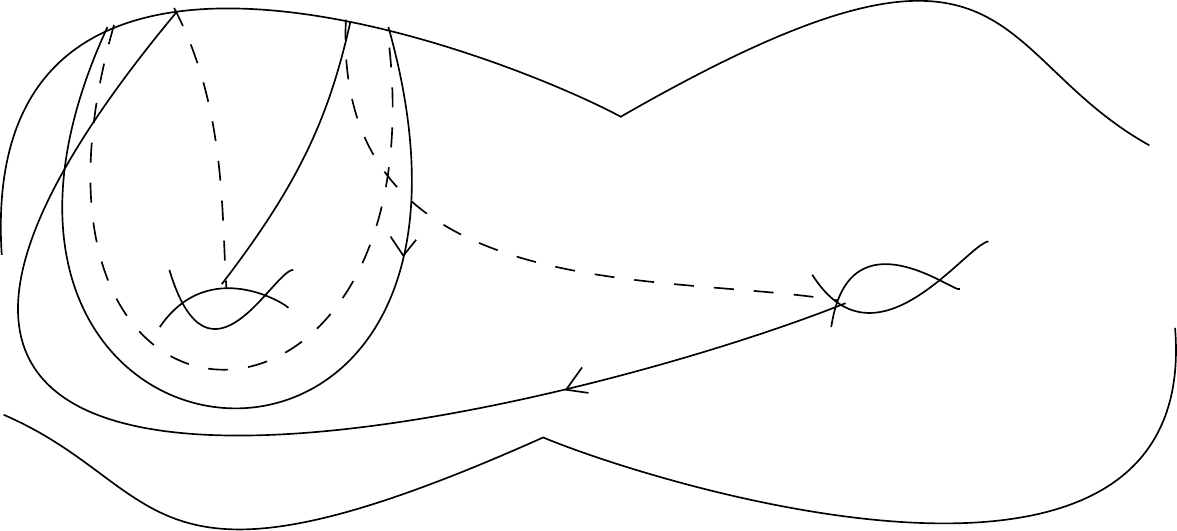} 

 \put(-135,10){\def\svgwidth{0.35\textwidth}
 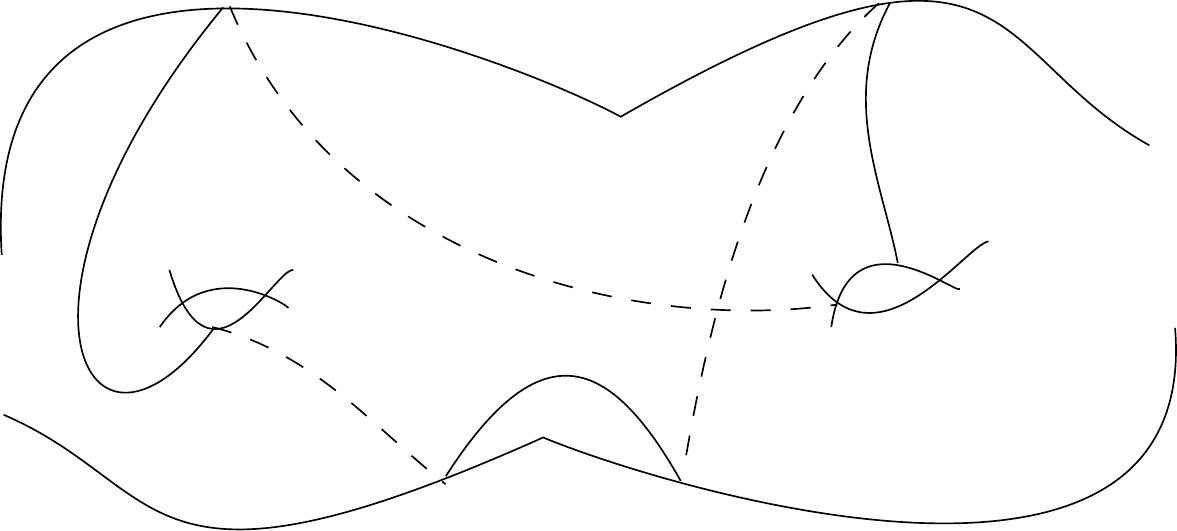} 

 \put(10,0){\def\svgwidth{0.5\textwidth}
 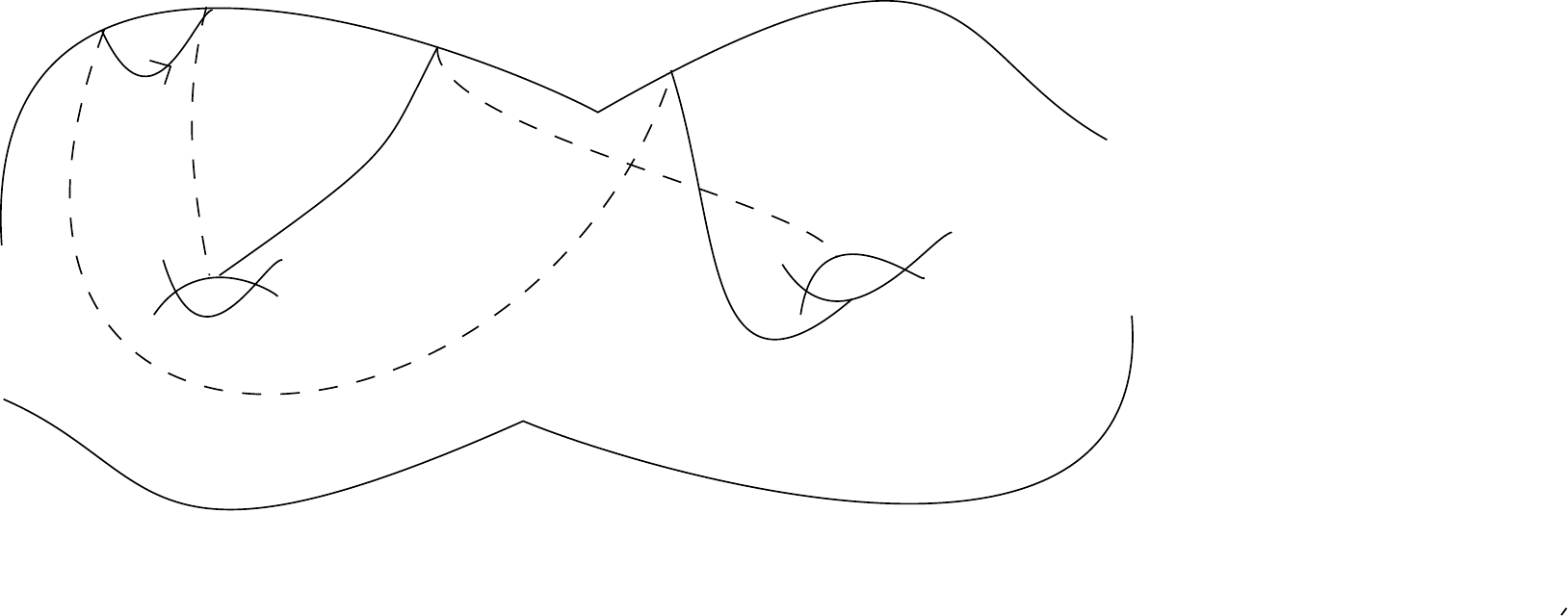}

 \put(-100,-5){$[\a_2]+[\g]+[\a_3]$}
 \put(30,-5){$-[\a_1]+[\g]+[\a_4]$}

\end{picture}

\caption{Pictorial proof for the case of one separating curve}
\label{fig:proof2}
\end{figure}

\noindent Case 3: Each $\a_i$ bounds a subsurface $\Sigma_i$. Now, if $\a_i$ bounds a torus, clearly $[\a_i]=T.$ If $\Sigma_i$ is a subsurface of higher genus, then one can decompose it into pairs of pants. One can show by induction that $$[\a_i]=\chi(\Sigma_i) T,$$ (indeed, as in Case 2, for each pair of pants having exactly one separating curve, the sum of classes of bounding curves equals $T$). By additivity of Euler characteristic and the next lemma, we can deduce the following $$[\a_1]+[\a_2]+[\a_3]=\chi(\Sigma) T+T=T.$$
\begin{figure}
\begin{picture}(0,90)

 \put(-85,0){\def\svgwidth{0.5\textwidth}
 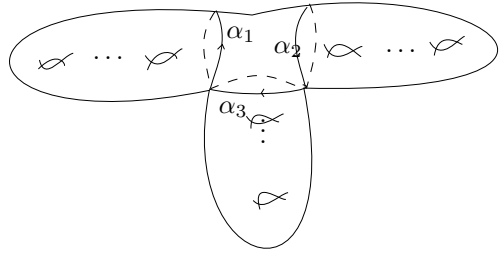} 

\end{picture}

\caption{Case 3: Each curve bounds a subsurface.}
\label{fig:each_separating}
\end{figure}

\end{proof}

\begin{lem}\label{subgroup lemma}
The subgroup of $K_0(Fuk(\Sigma))$ generated by $T$ maps isomorphically to $\mathbb{Z}/\chi(\Sigma)\mathbb{Z}$ under the winding number map. In particular,
$$\chi(\Sigma)T=0.$$
\end{lem}
\begin{proof} One can always find curves bounding a sub-surface with Euler characteristic between 0 and $\chi(\Sigma)-1.$ Moreover, one can always find a bounding curve with winding number equal to the Euler characteristic of this subsurface modulo $\chi(\Sigma)$. Hence the winding number map is surjective.
Let $\Sigma $ have even genus. Let there be a curve $\gamma$, that separates the surface into homeomorphic submanifolds. By the previous lemma we have the relation $$[\gamma]=-\frac{\chi(\Sigma)}{2} T.$$
Using Lemma \ref{T is well-defined}, if we consider a homeomorphism that permutes these two subsurfaces, then $\g=\g[1]$ and in the $K$-group, $$-\chi(\Sigma)T=[\g]+[\g[1]]=0.$$
If $\Sigma $ has odd genus, we consider a pair of separating curves to obtain the same relation. This relation gives us the desired result.
\end{proof}

\begin{prop}
If $\a_1$ and $\a_2$ are two homologous immersed curves such that $$wd_X(\a_1)=wd_X(\a_2)$$
then they represent the same class in $K_0(Fuk(\Sigma)).$\end{prop}
\begin{proof}
The Torelli group acts on our set of curves. This group is generated by Dehn twists about separating curves and by bounding twists. Let $\phi$ be an element of this group that maps $\a_1$ to $\a_2$. Using the previous lemma, it suffices to show that the difference between classes of $\a_1$ and $\a_2$ can be expressed as a multiple of $T$.
For genus two, there are no separating curves to define a bounding twist, so it is enough to consider $\phi $ to be a Dehn twist about a separating curve $\epsilon$. Thus $$[\a_2]=[\phi (\alpha_1)]=[\alpha_1],$$ where the second equality holds due to the non-empty intersection of $\epsilon$ with $\alpha_1$ and Lemma \ref{cone}. In the case of higher genus, each pair of bounding curves gives $$[\d_i]-[\d'_j]=\chi(\Sigma_j) T$$ where $\Sigma_j$ is the embedded subsurface bounded by $\d_j$ and $\d'_j$. Thus $\a_1$ and $\a_2$ have the same class in $K_0(Fuk(\Sigma))$ modulo $T$, considering their winding numbers are given to be equal.
\end{proof}
\begin{figure}
\begin{picture}(0,90)

 \put(-135,0){\def\svgwidth{0.8\textwidth}
 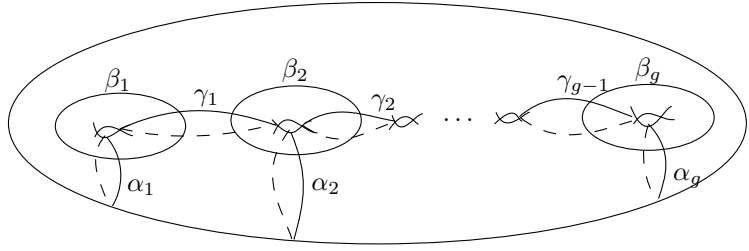}

\end{picture}

\caption{Lickorish generators of the mapping class group}
\label{fig:lickorish}
\end{figure}

\emph{Proof of Theorem 1.1.} We have constructed the following map $$K_0(Fuk(\Sigma))\to H_1(\Sigma)\oplus \mathbb{Z}/\chi(\Sigma)\mathbb{Z}$$
which was surjective to begin with. To show injectivity, we show that any twisted complex having a trivial cohomology class lies in the subgroup generated by $T$. The separating curves we know by Lemma \ref{pop lemma}, are mapped to the subgroup generated by $T$.  Assume that $[\Gamma]$ can be expressed as a sum of classes of non-separating curves. Since the mapping class group acts transitively on isotopy classes of non-separating curves, Proposition \ref{prop cones twisted} implies that the non-separating curves appearing in $[\Gamma]$ can be written as a linear combination of Lickorish generators \cite{lickorish1964finite} seen in Figure \ref{fig:lickorish}. If the curves $\a_i$ and $\b_i$ are a basis for $H_1(\Sigma,\mathbb{Z})$, and $\g_i$ bounds a pair of pants with two of these basis elements, then $$[\a_i]-[\a_{i+1}]+[\g_i]=T.$$ In particular, if $[\Gamma]$ represents the trivial class in homology, then we can decompose it as a linear combination of the class of curves bounding pairs of pants, which implies that $[\Gamma]$ indeed lies in the subgroup generated by $T$.\hspace*{5.8cm}$\square$

\section{Action of the mapping class group}

We define in this section a natural mapping class group action on $Fuk(\Sigma)$. The main result here is that this action is faithful.
\begin{defn}
Let $f\in MCG(\Sigma)$, this mapping class acts on the set of objects of $Fuk(\Sigma)$ in the natural manner. For a pair of curves $\a,\b \subset \Sigma$, the mapping class $f$ acts on the morphisms as follows, $$f\cdot HF(\a,\b)=HF(f(\a),\b).$$
\end{defn}
We begin by stating a well-known geometric interpretation of Floer homology.
\begin{lem}
Let $\a,\b$ be closed curves in $\S$, intersecting transversally. Then, $$\dim_{\mathbb{Z}}{HF(\a,\b)}=i(\a,\b).$$
\end{lem}

Here $i(\a,\b)$ is the geometric intersection number of $\a$ and $\b$. This is defined as the minimal intersection number between any two isotopic representatives $\a', \b'$ of curves $\a \& \b$ with no bigons occurring between them. When $\a$ and $\b$ are isotopic, their geometric intersection number due to admissibility, is at least two. One has the following simple lemma.

\begin{lem}
The Floer homology for curves $\a$ and $\b$ is trivial if and only if $\a \& \b$ are non-isotopic and have disjoint representatives.
\end{lem}
\begin{proof}
If two curves $\a\&\b$ on a closed surface are isotopic, they have two bigons forming between them, which occur each with a positive and a negative coefficient. The differential vanishes, and the rank of $HF(\a,\b)$ is 2. If $\a\& \b $ are disjoint, one can choose representatives with a non-trivial intersection. This gives only one bigon between two intersection points which are closed and are at the same time exact. Thus $HF(\a,\b)=\{0\}$.
\end{proof}
 We now prove faithfulness of the action of $MCG(\S)$ on the Fukaya category $Fuk(\S)$. Considering two admissible essential closed curves $\g_1$ and $\g_2$, the following lemma is key.
\begin{lem}
For curves $\g_1,\g_2$ as above, if the following condition for Floer homology holds
$$HF(f(\g_1),\g_2)=HF(\g_1,\g_2),$$ then the curve $f(\g_1)$ is isotopic to the curve $\g_1$, for some $f\in MCG(\S).$
\end{lem}
\begin{proof}
Let us take $\g_1$ to be the first homology generator $\a_1$ and let $\g_2$ be a curve $\g$ bounding the subsurface of genus one containing $\a_1.$
By the previous lemma we have $$HF(\a_1,\g)=\{0\},$$ which implies by hypothesis, $$HF(f(\a_1),\g)=\{0\}.$$
 If $\b_1$ is the dual curve to $\a_1$ among the first homology generators of $\Sigma$, we have $$\dim_{\mathbb{Z}}{HF(\a_1,\b_1)}=1.$$ This gives $$\dim_{\mathbb{Z}}{HF(f(\a_1),\b_1)}=1.$$ Together this implies that $f(\a_1)$ can be isotoped away from $\g$ and lies in the genus one part bounded by $\g$. Thus $f(\a_1)$ must be isotopic to the combination $\a_1+q\b_1$, where $q\in \mathbb{Z}.$ Let this combination be denoted by $\d_q$. In order to show that $f(\a_1)$ is isotopic to $\a_1$, it suffices to show that for $\a'_1$, which is a push-off of $\a_1$,
$$HF(f(\a_1),\a_1)\cong HF(\a'_1,\a_1)\neq HF(\d_q,\a_1) \,\,\,\,\,\,\,\,\,\mbox{if} \,\,\, q\neq 0.$$
The first isomorphism follows from the geometric intersection number, $i(f(\a_1),\a_1)=i(\a'_1,\a_1)=0.$ The second assertion is clear using the simple observation that $i(\d_q,\a_1)=q.$ We therefore obtain the desired isotopy between $f(\a_1)$ and $\a_1$ for a fixed $f\in MCG(\S).$
Similarly, a mapping class acts trivially on all essential closed curves $\g_1$ contributing to homology of the surface. This holds true, since we can always find a bounding curve $\g_2$ giving us an isotopy between $f(\g_1)$ and $\g_1$.
\end{proof}

Note that the above statement does not hold true for any separating curve. Indeed, two curves bounding homeomorphic subsurfaces are quasi-isomorphic but never isotopic. The above lemma gives a collection of curves $\a_1,\a_2, \ldots, \a_g\in H_1(\S)$ on which the mapping class group acts trivially. Apart from these curves the Lickorish generators $\g_1,\g_2,\ldots,\g_{g-1}$ are also considered in following paragraphs. We cut the surface along these curves in order to prove faithfulness of the action of $MCG(\S)$ on $Fuk(\S)$.
\begin{theorem}
Let $f\in MCG(\S)$ and suppose that the following isomorphism holds for any admissible essential closed curves $\g_1 \& \g_2$,
$$HF(f(\g_1),\g_2)\cong HF(\g_1,\g_2).$$
Then $f$ is isotopic to the identity in $MCG(\S).$
\end{theorem}
\begin{proof}
The idea is to locally reduce to a union of discs and apply the Alexander's trick.
Following the discussion preceding the theorem, we cut the surface along the curves $\a_1, \a_2,\ldots,\a_g$ on which $f$ acts in a trivial manner, obtaining a new surface $\S_0$. This $\S_0$ is in fact a disc with boundary and $2g-1$ holes.
\begin{figure}
\begin{picture}(0,90)

 \put(-120,0){\def\svgwidth{0.63\textwidth}
 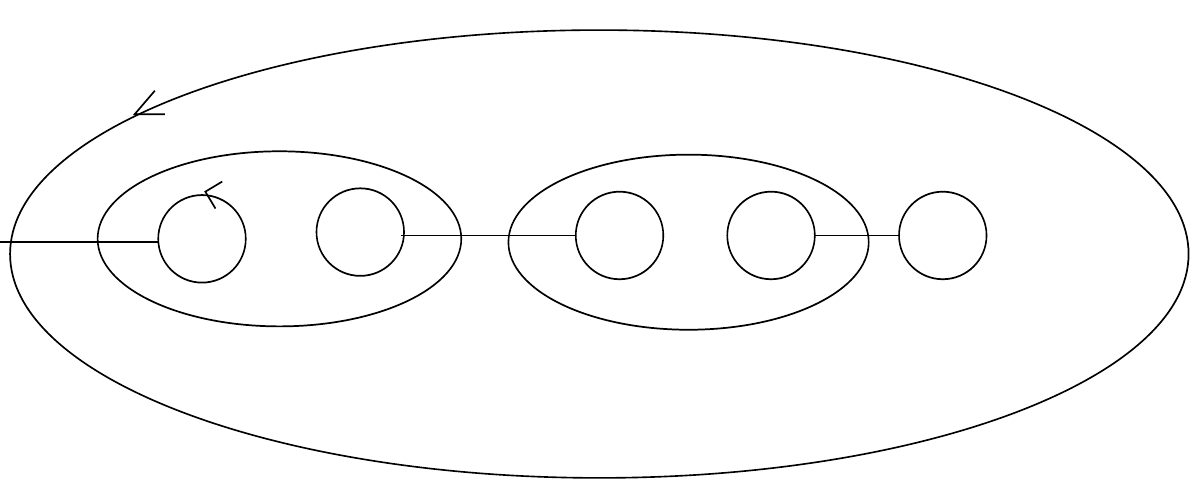} 

\end{picture}

\caption{Cutting along curves $\a_1, \a_2,\ldots,\a_g$}
\label{fig:S_0}
\end{figure}

In $\S_0$, the dual curves $\b_1,\b_2, \ldots, \b_g$ in $H_1(\S)$ can be seen as arcs joining consecutive holes and one joining a hole to the boundary. The remaining Lickorish generators, namely $\g_1,\g_2,\ldots,\g_{g-1}$, encircle pairs of holes in $\S_0$ (see Figure \ref{fig:S_0}). The mapping $f$ restricts to a mapping class $f_0$ in $(\S_0,\partial\S_0)$ fixing the boundary pointwise. Since $f$ acts trivially on the collection of curves $\b_1,\ldots,\b_g,\g_1,\ldots,\g_{g-1}$, by the previous lemma,
$f_0$ also acts trivially on this collection.

\begin{figure}
\begin{picture}(0,90)

 \put(-85,0){\def\svgwidth{0.5\textwidth}
 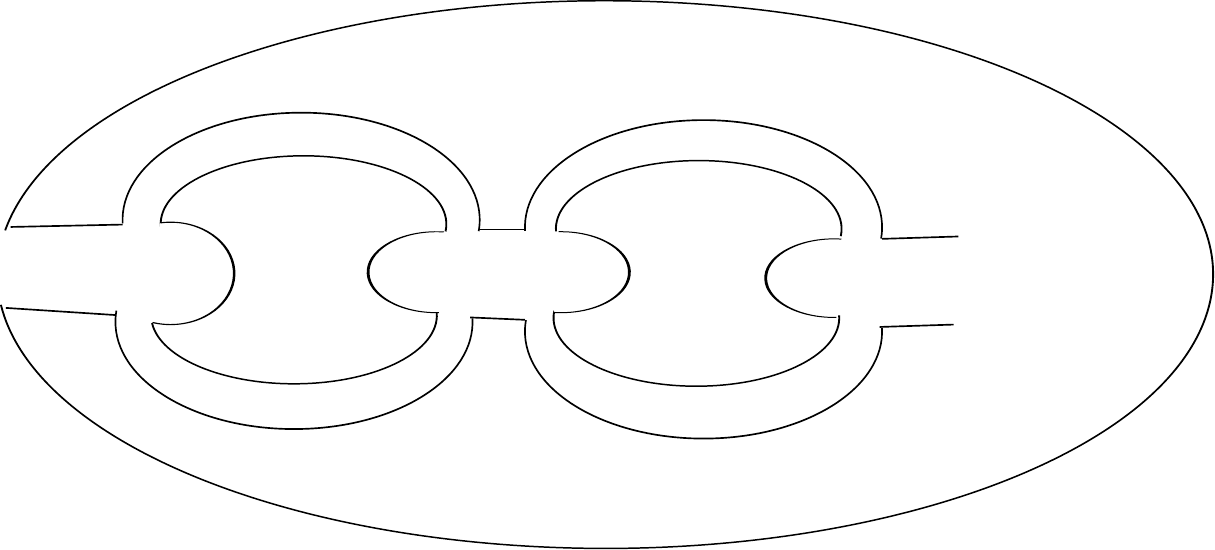} 
\end{picture}

\caption{Cutting along curves $\b_1,\ldots, \b_g,\g_1,\ldots,\g_{g-1}$}
\label{fig:S_1}
\end{figure}
Repeating the same argument as above, we cut along all of these curves (see Figure \ref{fig:S_1}). As a result we obtain a finite union of discs. These are $g$ in number. The mapping class $f_0$ now restricts to $f_1$ acting on this union and fixing the boundary (which is a union of curves) pointwise.
Now each disc in this union has a trivial mapping class group and we are done.
\end{proof}
\begin{appendix}
\section{Maslov index and rigid bigons}\label{appendix_bigons}
The boundary map in lagrangian Floer theory counts rigid holomorphic discs with prescribed boundary conditions.
 Here rigid means that up to parametrisation we get isolated solutions, which is controled by a Maslov index.
 We will give a winding number formula for this Maslov index and show that this counting coincides
 with our bigons in section \ref{Floer}.
 
Let
$(\g_1,\g_2)$ be an admissible pair of curves in the surface $\S$.
For points $p,q\in \g_1\cap \g_2$, a Whitney disc $u$ from $p$ to $q$
is a smooth map from 
the half disc \\$ \mathbb{D}=\{z\in \mathbb{C}\ ,\ |z|=1 \text{ and } \Im(z)\geq 0\}$ to $\S$, which is mapped as follows:
\begin{itemize}
\item
the point $1,-1\in \mathbb{D}$ are mapped to the points $p,q$ respectively,
\item
the clockwise and anti-clockwise boundary paths from $-1$ to $1$ are mapped to subsets of $\g_1$ and $\g_2$ respectively.
\end{itemize}
 The expected dimension of the space of holomorphic discs homotopic to $u$ is the Maslov index $\mu(u)$.
Let $X$ be a non singular vector field on $\S \setminus z$, where $z\in \S\setminus (\g_1\cup\g_2)$ is a generic point for $u$.  Using $X$, the tangent space of $\S\setminus z$ is trivialized. 
We denote by $n_z(u)$ the multiplicity of $u$ at $z$, and by $m_X(u)$ the degree of the loop in $\RP1$ defined in the trivialisation by the tangent line traversing a connected path along the two arcs of the boundary of $u$ in $\RP1$ avoiding the line $\delta$. Here the Maslov cycle $\delta\in \RP1$ is an arbitrary line distinct from the tangent vectors at corners.

\begin{prop}
The Maslov index is given by the formula
$$\mu(u)=m_X(u)+(4-4g)n_z(u)
\ .$$
\end{prop}
\begin{proof}
Following \cite{viterbo,floer1988index,robbin_salamon} the Maslov index $\mu(u)$ is equal to the Viterbo-Maslov index which can also be defined as  the degree of a loop in $\RP1$, but in a trivialisation of  the pull back $u^*(T\S)$. Let $p: \widetilde \S \rightarrow \S$ be the universal cover. We consider the lifting of $u$, $\tilde u: \mathbb{D}\rightarrow \widetilde \S$, $u=p\circ \tilde u$ of $u$. 
The tangent space $T\widetilde \S$ being trivial, the Maslov index $\mu(u)=\mu(\tilde u)$ is equal to $m_{\tilde Y}(\tilde u)$ where $\tilde Y$ is a non trivial vector field on
$\widetilde \S$, and $m_Y(\tilde u)$ is the degree in $\RP1$ analogous to $m_X(u)$. Let $\tilde X=p^*(X)$ be the pull back of $X$.
The difference $m_{\tilde X}(\tilde u)- m_{\tilde Y}(\tilde u)$ comes from the singularities of $\tilde X$ at points in $p^{-1}(z)$ where  each one contributes $2(2g-2)$.
We finally get $$\mu(u)=\mu(\tilde u)=m_{\tilde Y}(\tilde u)=m_{\tilde X}(\tilde u)+(4-4g)n_z(u)=m_X(u)+(4-4g)n_z(u)\ .$$
Note that when $u$ is a small embedded bigon around $z$, we have $m_X(u)=4g-3$ which gives the expected value $\mu(u)=1$.
\end{proof}
The boundary map in a Floer complex $CF(\gamma_1,\gamma_2)$ counts up to reparametrisation holomorphic discs with required boundary condition and Maslov index equal to $1$.
It is proved in \cite[Ch12]{de2014combinatorial} that for embedded  curves $\gamma_1$, $\gamma_2$ this counting is equal to the combinatorial one.
Using the universal cover we extend this result to our settings. 
\begin{prop}\label{Fbigon}
If  $\gamma_1$, $\gamma_2$ are admissible unobstructed immersed curves, then for any generators $p,q\in CF(\gamma_1,\gamma_2)$
the boundary of the Floer complex is combinatorial and given by the number of immersed bigons.
\end{prop}
\begin{proof}
If a disc $u: \mathbb{D}\rightarrow \S$ with required boundary has Maslov index equal to $1$ then its lifting $\tilde u$ to the universal  cover also has Maslov
index equal to $1$. Theorem 12.1 in \cite{de2014combinatorial} states that this is equivalent to $\tilde u$ being homotopic to an immersed bigon (a lune in their terminology).
This shows that all discs involved in the Floer boundary map correspond to immersed bigons. 
Conversely, by the Riemann mapping theorem every bigon has a unique holomorphic representative up to reparametrisation.
\end{proof}

\section{Rigid polygons}\label{appendix_polygons}
The previous result extends to polygons.
Given an admissible sequence of $n$ curves, $n\geq 2$, $\gamma_1,\dots,\gamma_n$, the corresponding higher product in Floer-Fukaya theory is defined by counting rigid holomorphic polygons with required boundary condition. The rigidity is also controlled by the index of the Cauchy-Riemann operator with lagrangian boundary conditions, which should be equal to $3-n$.
 For polygons with strip ends, the relation with the Maslov index
is given in \cite[Lemma 11.7]{seidel2008fukaya}. Following \cite{sarkar}, a combinatorial formula for the index is 
$$Ind(u)=2e(u)-\frac{n-2}{2}\ ,$$
where $e(u)$ is the Euler measure.
\begin{prop}
The number of rigid holomorphic $n$-gons is equal to the number of immersed convex $n$-gons.
\end{prop}
\begin{proof}
By the Riemann mapping theorem, each immersed convex $n$-gon is represented by a unique holomorphic map with the required boundary conditions. It remains to show that there are no other rigid holomorphic polygons. Let $u: \mathbb{D}\to \S$ be a holomorphic $n$-gon. This has a non-convex corner if and only if two adjacent curves on the boundary go inside the $n$-gon. In this situation we can build a $1$-dimensional space of homotopic holomorphic discs with the same prescribed boundary conditions. It follows that rigidity implies convexity. 

Moreover, rigid implies immersed. If we consider the surface $S$ built,  by identification on the boundary, from the components of $\S\setminus (\cup_i\gamma_i)$ with multiplicities prescribed by $u$. The map $u$ factors through $S$ which is an abstract disc. Then by the Riemann mapping theorem $\mathbb{D}$ is biholomorphic to $S$. This equivalence composed with the canonical immersion $S\to \S$ is the map $u $ which as a result is also an immersion.

\end{proof}

\end{appendix}

\bibliographystyle{amsalpha}
\bibliography{references}

\end{document}